\documentclass[12pt]{article}

\usepackage{amssymb}
\usepackage{fnpct}
\usepackage{amsmath}
\usepackage[pdftex]{graphicx}
\usepackage{epsfig}
\usepackage{framed}
\usepackage{etoolbox}
\usepackage{fnpct}
\usepackage{overpic}
\usepackage{url}
\usepackage[font=small,labelfont=bf]{caption} 
\usepackage{relsize} 
\usepackage{setspace}
\usepackage{hyperref}

\usepackage{tikz}
\usepackage{pdfpages}
\usetikzlibrary{arrows,positioning}
\usepackage{enumitem}

\usepackage{subfig}

\usepackage{tabularx}

\usepackage[utf8]{inputenc}

\newcommand{\Z}{\mathbb{Z}}
\newcommand{\diag}{{\rm diag}}
\newcommand{\SNR}{{\rm SNR}}
\newcommand{\Cond}{{\rm Cond}}

\newcommand{\spa}{{\rm span \,}}

\newcommand{\N}{\mathbb{N}}

\newcommand{\C}{\mathbb{C}}

\newcommand\numberthis{\addtocounter{equation}{1}\tag{\theequation}}

\DeclareMathOperator*{\Tr}{Tr}

\def\hide #1 {}
\long\def\longhide #1 {}

\usepackage{amsthm}

\theoremstyle{plain}
\newtheorem{theorem}{Theorem}[section]
\newtheorem*{theorem*}{Theorem}
\newtheorem{lemma}[theorem]{Lemma}
\newtheorem{proposition}[theorem]{Proposition}

\newtheorem{corollary}[theorem]{Corollary}

\theoremstyle{definition}
\newtheorem{remark}[theorem]{Remark}
\newtheorem{example}[theorem]{Example}

\newtheorem{definition}[theorem]{Definition}

\theoremstyle{definition}

\newcommand*{\Scale}[2][4]{\scalebox{#1}{$#2$}}%

\newcounter{nootje}
\renewcommand{\check}[1]
  {\marginpar{\tiny\begin{minipage}{20mm}\begin{flushleft}\thenootje : #1\end{flushleft}\end{minipage}}\addtocounter{nootje}{1}}
\usepackage[backend=biber]{biblatex}
\title{Robustness to noise and erasures of Gabor frames\\ in finite dimensions}

\author{\begin{tabular}{ccc}
  Palina Salanevich\thanks{Mathematical Institute, Utrecht University. {\sl Email}: p.salanevich@uu.nl} & \quad \quad \quad \quad & Nigel Q. D. Strachan\thanks{Mathematics Department, Utrecht University; Department of Intelligent Systems, Delft University of Technology (from September 1 2025). {\sl Email}: n.q.d.strachan@uu.nl/tudelft.nl }
  \end{tabular}
}

\begin{document}
\maketitle
\date{}

\begin{abstract}
In this paper, we investigate the robustness of structured frames to measurement noise and erasures, with the focus on Gabor frames $(g, \Lambda)$ with arbitrary sets of time-frequency shifts $\Lambda$. This property of frames is important
in many signal processing applications, from wireless communication to phase retrieval and quantization.  
We aim to analyze the dependence of the frame bounds of such frames on their structure and cardinality and provide constructions of nearly tight Gabor frames with small $\vert \Lambda \vert$.
We show that the frame bounds of Gabor frames with a random window $g$ and frame set $\Lambda$ show similar behavior to the frame bounds of random frames with independent entries. 
Moreover, we study uniform estimates for the frame bounds in the case of measurement erasures. We prove numerical robustness to erasures for mutually unbiased bases frames with erasure rate up to 50\% and show that the Gabor frame generated by the Alltop window can provide results similar to the best previously known deterministic constructions.

\end{abstract}

\section{Introduction}
 
Frames generalize the classical notion of a basis and provide redundant representations of signals. They became a powerful tool in many areas of applied mathematics, computer science, and engineering. Among other applications, frames are used in communication systems for signal transmission over a noisy channel~\cite{fickus2012numerically}; and also in image processing~\cite{kingsbury1998wavelet}, phase retrieval~\cite{jaganathan2016phase}, tomography~\cite{innocenti2023shadow}, and speech recognition~\cite{balan1}, where the initial signal is not available and we have access to its measurements in the form of frame coefficients instead.

Formally, a set of vectors $\Phi = \{\varphi_j\}_{j = 1}^N \subset \mathbb{C}^M$ is called a \emph{frame} if $\Phi$ spans the ambient space $\mathbb{C}^M$. One usually considers \emph{overcomplete} frames for which $N > M$.
Any signal $x\in \mathbb{C}^M$ can be uniquely represented by its \emph{frame coefficients} $\{\langle x, 
\varphi_j\rangle\}_{j = 1}^N$. 
A key advantage of such a frame representation is its redundancy. 
It allows one to ensure that a signal can still be reconstructed if a part of its frame coefficients is lost in the measurement or transmission process. Furthermore, in many applications, the frame coefficients are corrupted by measurement noise. In these cases, redundancy of the frame representation allows to mitigate it and obtain a stable signal reconstruction.

An important property of a frame that determines how stable is the signal reconstruction from its noisy frame coefficients is the lower and upper frame bounds $A_{\Phi}$ and $B_{\Phi}$. They are defined as the optimal constants such that for all $x \in \mathbb{C}^M$,
\begin{align*}
    A_{\Phi} \|x\|_2^2 \leq \sum_{j = 1}^N |\langle x, \varphi_j\rangle|^2 \leq B_{\Phi} \|x\|_2^2.
\end{align*}
When the frame bounds of a frame $\Phi$ are sufficiently close to each other, we call $\Phi$ \emph{well-conditioned}. We postpone a more detailed discussion of the frame theory background until Section~\ref{sec: frame theory background}. 

Frame bounds for random frames with independent entries have been well studied~\cite{vershynin2018high, rudelson1, latala2005some}. However, in many signal processing problems, the structure of the frame is dictated by the specific application. This motivates the study of structured (and random structured) frames. 

In this paper, we focus on the \emph{Gabor frames} that arise naturally in imaging, microscopy, and audio processing applications~\cite{bammer2019gabor, rolland2010gabor}.
A Gabor frame $(g, \Lambda)$ is generated by a single vector $g \in \mathbb{C}^M$ called \emph{window} and a \emph{frame set} $\Lambda \subset \mathbb{Z}_M \times \mathbb{Z}_M$ by taking time-frequency shifts $\pi(\lambda)$ of $g$ for $\lambda \in\Lambda$, see Section \ref{sec: frame theory background} for a precise definition. One can randomize this construction by considering a random window $g$. However, each vector in $(g, \Lambda)$ is a unitary transformation of $g$, and thus the vectors in a random Gabor frame are heavily dependent. Therefore, different techniques are needed to analyze the frame bounds for random Gabor frames. Furthermore, the frame bounds of $(g, \Lambda)$ depend not only on the cardinality $\vert \Lambda \vert$ of the frame (as in the case of random frames with independent vectors), but also on the structure of $\Lambda$. We analyze the frame bounds of Gabor frames for $\Lambda$ with a specific structure and also for generic $\Lambda$. In the latter case, we generate $\Lambda$ at random, from the uniform distribution on all $\Lambda\subset \mathbb{Z}_M\times \mathbb{Z}_M$ with a given cardinality. For a generic $\Lambda$, we obtain frame bounds that are comparable to those of random frames with independent vectors. 

We also investigate robustness of structured frames to adversarial measurement erasures. To guarantee stable reconstruction of the signal in the case of $k$ erasures, one needs to derive a uniform bound on the frame bounds of all subframes of cardinality $N-k$. We show robustness to up to $N/2$ erasures for a large class of \emph{mutually unbiased bases frames}. This class of frame is characterized by its `good' global structure. From this result, we derive an example of a deterministic  Gabor frame that is robust to both noise and high-rate erasures.  

\subsection{Related work}
Frame bounds have been well studied for random frames with independent vectors. The standard estimates of Gaussian and, more generally, subgaussian frames can be found in~\cite{vershynin2018high}. An optimal estimate for the lower frame bound of a subgaussian random frame was obtained by  Ruldeson and Vershynin in \cite{rudelson2}. They showed that with high probability it is bounded  from below by $c\left(\sqrt{\frac{N}{M}} - \sqrt{\frac{M-1}{M}}\right)$. An estimate for the upper frame bound was derived by Latała for a general class of random frames with independent bounded fourth moment entries~\cite{latala2005some}. More specifically, Latała showed that for such frames $B_{\Phi} \leq C\frac{N}{M}$ with high probability.  

To the best of our knowledge, no results on the frame bounds of Gabor frames with general $\Lambda\subset \mathbb{Z}_M\times \mathbb{Z}_M$ were obtained before. However, a closely related result on the singular values of the matrix associated with an \emph{underdetermined} Gabor system $(g,\Lambda)$ with a random window $g$ and $|\Lambda| < M$ has been established in \cite{pfander2010sparsity}. This result is motivated by the compressive sensing problem and aims to evaluate the \emph{restricted isometry property} of time-frequency structured random matrices, see also~\cite{krahmer2014suprema}. The result from~\cite{pfander2010sparsity} is formulated as follows.

\begin{theorem}\emph{ \cite{pfander2010sparsity}}
Let $\varepsilon, \delta \in (0,1)$ and consider a random window $g$ given by ${g(j) = \frac{1}{\sqrt{M}}e^{2\pi i y_j}}$, $j\in \mathbb{Z}_M$, with $y_j$ independent uniformly distributed on~$[0,1)$. For a Gabor system $(g, \Lambda)$, consider the matrix $\Phi_{\Lambda}$ that has vectors of $(g, \Lambda)$ as its columns. Then, for any $\Lambda\subset \mathbb{Z}_M\times \mathbb{Z}_M$ with
\begin{equation*}
|\Lambda| \le \frac{\delta^2 M}{4e (\log (|\Lambda|/\varepsilon) + c)}
\end{equation*}
\noindent for $c = \log(e^2 / (4(e - 1)))\approx 0.0724$, the minimal and maximal nonzero singular values of $\Phi_\Lambda$ satisfy $$\mathbb{P}\left( 1 - \delta\le  \sigma_{\min}^2(\Phi_\Lambda) \le  \sigma_{\max}^2(\Phi_\Lambda) \le  1+\delta\right) \geq 1 - \varepsilon.$$ 
\end{theorem}
\noindent In this paper, we prove analogous results for overcomplete Gabor frames with $|\Lambda| > M$.

\smallskip

In the second part of the paper, we discuss robustness of signal representations with structured frames under (adversarial) measurement erasures. The concept of Numerically Erasure-Robust Frames (NERF), which formalizes the idea of frames that are robust to both additive noise and erasures, was introduced by Fickus and Mixon in \cite{fickus2012numerically}. 
\begin{definition}[Numerically Erasure-Robust Frame \cite{fickus2012numerically}]\label{def: NERF}
For a fixed $p\in [0,1]$ and $C \geq 1$, a frame $\Phi = \{\varphi_j\}_{j = 1}^N$ is called a \emph{$(p,C)$-numerically erasure-robust frame} if, for every $J \subset \{1,\dots,N\}$ of cardinality $|J| = (1 - p)N$, the condition number of the analysis matrix of the corresponding subframe $\Phi_J = \{\varphi_j\}_{j \in J}$ satisfies $\Cond(\Phi_J^*) \leq C$.
\end{definition}

In other words, one can stably reconstruct a signal whenever \emph{any} $p$-fraction of measurements is removed. In~\cite{fickus2012numerically}, it was shown that random Gaussian frames with $N = O(M)$ have the NERF property with an erasure rate of at most 15\%. For deterministic frames, the strongest NERF property was shown for equiangular tight frames (ETF). 

\begin{theorem}[\cite{fickus2012numerically}]\label{th:ETF}
Let $\Phi = \{\varphi_j\}_{j = 1}^N$ be an equiangular tight frame such that ${\frac{N - 1}{M(M - 1)}\geq \alpha}$. Then $\Phi$ is a $(p,C)$-numerically erasure-robust frame for any $p \leq \frac{\alpha (C^2 - 1)^2}{\alpha(C^2 - 1)^2 + (C^2 + 1)^2}$.
\end{theorem}
Informally, this result states that we can stably reconstruct a signal from its ETF frame coefficients for the erasure rate at most 50\%. Note that the cardinality of numerically erasure-robust ETFs can be rather high, with $N = O(M^2)$. Fickus and Mixon also obtained results on the numerical robustness to erasures for another class of \emph{mutually unbiased bases frames}~\cite{fickus2012numerically}. We discuss it in more details in Section~\ref{sec: deterministic NERF}, where we improve this result.

\subsection{Main results}
In the first part of the paper, we investigate the frame bounds for Gabor frames $(g, \Lambda)$ with a random window $g$. As the properties of such a frame depend not only on the properties of $g$, but also on the structure of $\Lambda$, we start by estimating frame bounds for regularly structured $\Lambda = F \times \mathbb{Z}_M$ or $\Lambda = \mathbb{Z}_M \times F$ for an arbitrary $F\subset \mathbb{Z}_m$ in Proposition \ref{prop_sing_val_Gabor_regular}. For these frames, we show that the frame bounds are determined by how ``spiky'' the window $g$ is.

For $\Lambda$ with arbitrary structure, we develop an approach for reducing frame bounds estimation to a combinatorial problem in Lemma~\ref{lemma_trace_formula}. We apply this generic result to estimate the frame bounds for Gabor frames with Gaussian and Steinhaus random windows in Corollary~\ref{cor: m2 Steinhaus general} and Corollary~\ref{cor:gausstrace}. The obtained estimates depend only on the cardinality $\vert \Lambda \vert$ and for $N = O(M^2)$ have optimal scale $\frac{N}{M}$ that is attained by tight frames. 

To overcome the restriction $N = O(M^2)$ and study frame bounds for Gabor frames with smaller cardinality, we consider generic frame sets $\Lambda$.
Our main result in this part of the paper shows that Gabor frames with random window and random time-frequency set $\Lambda$ attain estimates similar to the Gaussian case with high probability.

\begin{theorem}\label{th_sing_val_rand_lambda}
Let $g$ be a random window given by ${g(j) = \frac{1}{\sqrt{M}}e^{2\pi i y_j}}$, $j\in \mathbb{Z}_M$, with $y_j$ independent uniformly distributed on~$[0,1)$. For any fixed even $m\in \mathbb{N}$, consider a Gabor system $(g, \Lambda)$ with a random set $\Lambda\subset \mathbb{Z}_M\times \mathbb{Z}_M$ constructed so that events $\{(k,\ell)\in \Lambda\}$ are independent for all $(k,\ell)\in \mathbb{Z}_M\times \mathbb{Z}_M$ and have probability $\tau = \frac{C\log M}{M^{\frac{m-1}{m}}}$, where $C>0$ is a sufficiently large constant. Then, with high probability (depending only on $m$, $\delta$, and $C$),
\begin{equation*}
\frac{|\Lambda|}{M}(1-\delta) \le A_{(g, \Lambda)} \le B_{(g, \Lambda)}\le \frac{|\Lambda|}{M}(1+\delta).
\end{equation*}
\end{theorem}

This probabilistic (in terms of random choice of $\Lambda$) result can be interpreted as follows. While there could be ``bad'' choices of $\Lambda$ that lead to suboptimal frame bounds, most of the subframes $(g, \Lambda)\subset (g, \mathbb{Z}_M \times \mathbb{Z}_M)$ with $\vert\Lambda\vert = O(M^{1+\eta})$ for $\eta$ arbitrary small are well-conditioned.

\medskip

In the second part of the paper, we investigate NERF properties of Gabor frames. We show that the Gabor frame generated by the Alltop window has robustness to erasures similar to ETFs. This result follows from a more general result we prove for mutually unbiased bases (MUB) frames (see Section \ref{sec: deterministic NERF} for the definition). These frames were previously shown to be only weakly NERF in \cite{fickus2012numerically}, where the erasure rate scales with $\frac{1}{M}$ for the asymptotically maximal MUB of size $M^2$. In Theorem~\ref{th:MUBGuarantee}, we show that MUB frames are robust to up to 50\% erasures. Here, we formulate the result for Gabor frames, which is a corollary of Theorem~\ref{th:MUBGuarantee}.
\begin{theorem}
    Let $M \geq 5$ be prime and let $g_A$ with $g_A(j) = \frac{1}{\sqrt{M}}e^{2 \pi i j^3/ M}$ be the \emph{Alltop window}. Then the frame $(g_A, F \times \Z_M)$ with arbitrary $|F| = \alpha M$ is a $(p,C)$-numerically erasure robust frame for any $p\leq \frac{\alpha (C^2 - 1)^2}{\alpha(C^2 - 1)^2 + (C^2 + 1)^2}$.
\end{theorem}
Note that similarly to Theorem~\ref{th:ETF} for ETFs, Gabor frames can provide stable reconstruction with additive noise and adversarial erasures up to 50\%.

\subsection{Notation}
In this section, we briefly introduce the notation that will be used throughout the paper.
Let $\mathbb{S}^{M-1} = \{x\in \mathbb{C}^M,  \Vert x\Vert _2 = 1\}$ denote the complex unit sphere. For a matrix $A$, let us denote the smallest and largest singular values of $A$ by $\sigma_{\min}(A)$ and $\sigma_{\max}(A)$ respectively.
We view a vector $x\in\mathbb{C}^M$ as a function $x:\mathbb{Z}_M\to~\mathbb{C}$, that is, all the operations on indices are done modulo $M$ and $x(m-k) = x(M+m -k)$. We define the normalized discrete Fourier transform matrix as $\mathcal{F}_M = \frac{1}{\sqrt{M}} \left( e^{-2\pi i k\ell/M} \right)_{k,\ell \in \Z_M}$. The identity $M\times M$ matrix is denoted by $I_M$.

Throughout the paper, we consider the following types of random vectors. \emph{Steinhaus} random vector $g\in \mathbb{C}^M$ is defined so that ${g(m) = \frac{1}{\sqrt{M}}e^{2\pi i y_m}}$, $m\in \mathbb{Z}_M$, and $y_m$ are independent uniformly distributed on~$[0,1)$. We denote a complex \emph{Gaussian} random vector distribution by  $g\sim \mathcal{C}\mathcal{N}\left( \mu, \Sigma\right)$, where $\mu = \mathbb{E}(g)$ and $\Sigma$ denotes the covariance matrix. Finally, we write $g\sim\text{Unif}(\mathbb{S}^{M-1})$ for a random vector uniformly distributed on the complex unit sphere.

\bigskip

The remaining part of the paper is organized as follows. We provide an overview of relevant frame theoretic background in Section \ref{sec: frame theory background}. In Section \ref{sec: frame bounds general}, we develop general analysis tools for frame bounds estimation, which we then apply to obtain our results for random and deterministic Gabor frames in Section~\ref{sing_val_sec}. 
We formulate and prove numerical erasure-robustness for MUB frames and deterministic Gabor frames in Section \ref{sec: deterministic NERF}. Finally, Section \ref{sec_num_res_sing_val} contains numerical results on the frame bounds of random Gabor frames with small cardinality and their robustness to erasures. There, we also discuss open problems and the direction for further research. The Appendix contains the probabilistic tools and results used in this paper.

\section{Frame theory background}\label{sec: frame theory background}
Before we dive into the evaluation of Gabor frame bounds, we discuss the necessary background on frame theory, and Gabor frames in particular. For a comprehensive background on frames in finite dimensions, we refer the reader to \cite{finite_frames_book}.

\begin{definition}\label{def: frame}
  Let $\mathcal{H}$ be a Hilbert space. A set of vectors ${\Phi = \{\varphi_j\}_{j \in J}\subset \mathcal{H}}$ is called a \emph{frame} with \emph{frame bounds} $0< A \leq B$ if, for any $x\in \mathcal{H}$, the following inequalities hold
\begin{equation}\label{eq: frame_def_intro}
A\Vert x\Vert^2\le \sum_{j \in J} \vert\langle x, \varphi_j\rangle\vert^2 \le B\Vert x\Vert^2.
\end{equation}
\end{definition}

In this paper, we focus exclusively on the finite dimensional setup, that is, when $\mathcal{H} = \mathbb{C}^M$. Note that in this case, the inequality~\eqref{eq: frame_def_intro} holds for some $0< A \leq B<\infty$ if and only if ${\spa (\Phi) = \mathbb{C}^M}$. That is, in the finite dimensional case, the notion of a frame is equivalent to the notion of a spanning set. In particular, we have $\vert\Phi\vert = N\ge M$. 

By a slight abuse of notation, we identify a frame $\Phi = \{\varphi_j\}_{j = 1}^N\subset \mathbb{C}^M$ with its \emph{synthesis matrix} $\Phi$, having the frame vectors $\varphi_j$ as its columns. The adjoint $\Phi^*$ of the synthesis matrix is called the \emph{analysis matrix} and the product $\Phi \Phi^*$ is called the \emph{frame operator} of the frame $\Phi$. 
The values $\langle x, \varphi_j\rangle$, $j\in \{1,\dots, N\}$, are called the \emph{frame coefficients} of $x$. The vector of frame coefficients can be written as $\Phi^* x$.

To reconstruct a vector from its frame coefficients, one can use a \emph{dual frame} $\tilde{\Phi} = \{\tilde{\varphi}_j\}_{j = 1}^N$, defined so that $x = \sum_{j=1}^N \langle x, \varphi_j\rangle \tilde{\varphi}_j$, for each $x\in \mathbb{C}^M$.  A dual frame is not uniquely defined if $|\Phi|>M$. The \emph{standard dual frame} of $\Phi$ is given by the Moore-Penrose pseudoinverse $(\Phi \Phi^*)^{-1}\Phi$ of the synthesis matrix $\Phi$. 

Using the matrix notation we just introduced, the \emph{optimal frame bounds} can be obtained as 
\begin{equation}\label{eq: lower frame bound def}
    A_\Phi = \inf_{x\in \mathbb{C}^M\setminus \{0\}} \frac{\sum_{j = 1}^N \vert\langle x, \varphi_j\rangle\vert^2}{\Vert x\Vert_2^2} = \min_{x\in \mathbb{S}^{M-1}}\Vert \Phi^*x\Vert_2^2 = \sigma_{\min}^2(\Phi^*);
\end{equation}
\begin{equation}\label{eq: upper frame bound def}
    B_\Phi = \sup_{x\in \mathbb{C}^M\setminus \{0\}} \frac{\sum_{j = 1}^N \vert\langle x, \varphi_j\rangle\vert^2}{\Vert x\Vert_2^2} =  \max_{x\in \mathbb{S}^{M-1}}\Vert \Phi^*x\Vert_2^2 = \sigma_{\max}^2(\Phi^*).
\end{equation}

\smallskip

In the case when the frame coefficients are corrupted by noise (e.g., due to measurement error or quantization), frame bounds serve as a measure of the reconstructed signal distortion. Indeed, let $c = \Phi^*x+ \delta\in \mathbb{C}^N$ be a vector of noisy frame coefficients of a signal $x\in \mathbb{C}^M$ with respect to the frame $\Phi$, where $\delta\in \mathbb{C}^N$ is a noise vector. Then an estimate $\tilde{x}$ of the initial signal $x$ obtained from $c$ using the standard dual frame of $\Phi$ is given by
\begin{equation*}
\tilde{x} = (\Phi \Phi^*)^{-1}\Phi c = x + (\Phi \Phi^*)^{-1}\Phi\delta,
\end{equation*}
\noindent and the reconstruction error
\begin{equation*}
\Vert \tilde{x} - x\Vert _2^2 \le \Vert (\Phi \Phi^*)^{-1}\Phi \Vert _2^2\Vert \delta\Vert _2^2 = \frac{\Vert \delta\Vert _2^2}{\sigma_{\min}^2(\Phi^*)} = \frac{\Vert \delta\Vert _2^2}{A_\Phi}.
\end{equation*}
Moreover, for a given signal to noise ratio $\SNR = \frac{\Vert \Phi^* x\Vert _2}{\Vert \delta\Vert _2}$, the norm of the reconstruction error $\Vert (\Phi \Phi^*)^{-1}\Phi\delta\Vert _2$ compares to the norm of the initial signal $\Vert x\Vert _2$ as
\begin{equation*}
\frac{\Vert (\Phi \Phi^*)^{-1}\Phi\delta\Vert _2}{\Vert x\Vert _2}\le \frac{\Cond(\Phi^*)}{\SNR},
\end{equation*}
where
\begin{equation*}
\begin{split}
\Cond(\Phi^*) & = \sup_{x\in \mathbb{C}^M\setminus \{0\}}\sup_{\delta\in \mathbb{C}^N\setminus \{0\}} \SNR\frac{\Vert (\Phi \Phi^*)^{-1}\Phi\delta\Vert _2}{\Vert x\Vert _2} \\
& = \sup_{x\in \mathbb{C}^M\setminus \{0\}} \frac{\Vert \Phi^* x\Vert _2}{\Vert x\Vert _2} \sup_{\delta\in \mathbb{C}^N\setminus \{0\}}\frac{\Vert (\Phi \Phi^*)^{-1}\Phi\delta\Vert _2}{\Vert \delta\Vert _2} = \frac{\sigma_{\max}(\Phi^*)}{\sigma_{\min}(\Phi^*)} = \frac{\sqrt{B_\Phi}}{\sqrt{A_\Phi}}.
\end{split}
\end{equation*}
\noindent That is, $\Cond(\Phi^*)$ is the condition number of the analysis matrix of the frame~$\Phi$.

Thus, the ratio between the (optimal) frame bounds measure the robustness of the signal reconstruction from noisy frame coefficients, and the closer the frame bound are, the more \emph{well-conditioned} frame $\Phi$ is. In the case when $A_\Phi = B_\Phi$, the frame $\Phi$ is called \emph{tight}. Furthermore, in the case the frame vectors are normalized so that $\Vert \varphi_j\Vert_2 = 1$, for all $j\in \{1, \dots, N\}$, $\Phi$ is called a \emph{unit norm frame}.

\subsection{Coherence and Welch bounds}
In many signal processing scenarios, the ``quality'' of a frame is measured by how well-spread in space the frame vectors are. One way to formalize this is via the notion of \emph{frame coherence}. 

\begin{definition}[Coherence]
    Let $\Phi = \{\varphi_j\}_{j = 1}^N$ be a unit-norm frame. We define the \emph{coherence} $\mu(\Phi)$ of $\Phi$ as
    \begin{align*}
        \mu(\Phi) := \max_{j \neq j'} |\langle \varphi_j, \varphi_{j'} \rangle|.
    \end{align*}
\end{definition}

We refer to frame coherence when addressing numerical robustness to erasures for a specific class of frames in Section~\ref{sec: deterministic NERF}.
A related notion is \emph{frame potential} defined in~\cite{benedetto2003finite}. 

\begin{definition}[Frame potential]
Let $\Phi = \{\varphi_j\}_{j = 1}^N$ be a frame. We define the \emph{frame potential} $\mathrm{FP}(\Phi)$ of $\Phi$ as 
\begin{align*}
    \mathrm{FP}(\Phi) := \|\Phi^* \Phi\|_{\mathrm{HS}}^2 =  \sum_{j = 1}^N \sum_{j' = 1}^N |\langle \varphi_j, \varphi_{j'}\rangle|^2.
\end{align*}
\end{definition}
The following proposition gives a general lower bound as well as a characterization of tight frames. We use it to prove a general lower bound on the lower frame bounds $A_\Phi$ of unit-norm frames $\Phi$ in Section~\ref{sec: deterministic NERF}.

\begin{proposition}[Zero-th order Welch bound \cite{strohmer2003grassmannian}]\label{pr:zerowelch}
Let $\Phi = \{\varphi_j\}_{j = 1}^N$ be a unit-norm frame. Then we have that
\begin{align*}
    \mathrm{FP}(\Phi) \geq \frac{N^2}{M}. 
\end{align*}
Moreover, equality is achieved if and only if $\Phi$ is a tight frame. 

\end{proposition}

\subsection{Gabor frames} \label{sec:gabor}
The purpose of this section is to introduce Gabor frames and some of their basic properties. For a more detailed discussion, the reader is referred to~\cite{pfander2}.

\begin{definition}[Gabor frame] For a \emph{window} $g \in \C^M \setminus \{0\}$ and $\Lambda \subset \Z_M \times \Z_M$ we define the \emph{Gabor system} generated by \emph{window} $g$ and set $\Lambda$ as
    \begin{align*}
        (g,\Lambda) = \{\pi(\lambda)g : \lambda \in \Lambda\}.
    \end{align*}
Here, the \emph{time-frequency shift} operator $\pi(k,\ell) = M_\ell T_k$ is a composition of the \emph{time shift} by $k$ defined as $T_kx = (x(j - k))_{j \in \Z_M}$, and the \emph{frequency shift (modulation)} by $\ell$ defined as $M_{\ell}x = (e^{2 \pi i \ell j / M} x(j))_{j \in \Z_M}$.
In the case that $(g,\Lambda)$ spans $\C^M$, we call $(g,\Lambda)$ a \emph{Gabor frame}.
\end{definition}

    A useful property of the time-shift and frequency-shift operators is that they commute up to a global phase factor~\cite{pfander2}. 
    \begin{lemma}\label{lm:timeshiftfrequencyshiftcommuteuptophase}
    For $\lambda, \mu \in \Z_M \times \Z_M$ we have that
    \begin{align*}
        \pi(\lambda)\pi(\mu) = c_{\lambda,\mu} \pi(\lambda + \mu) = c_{\lambda,\mu} \overline{c_{\mu,\lambda}}\pi(\mu)\pi(\lambda)
    \end{align*}
    where $c_{(k,\ell), (k',\tilde{\ell})} = e^{2\pi i k \tilde{\ell}}$. In particular, we have that 
    
    \begin{align*}
         \pi(\lambda)^* = \pi(\lambda)^{-1} = c_{\lambda, \lambda}\pi(-\lambda)
    \end{align*}
\end{lemma}

Furthermore, many properties of the time-frequency shift operators are determined by their close relation to the Fourier transform. As such, this relation can be used to show the following result.
\begin{proposition}[Fourier duality]\label{prop: Fourier duality}
    Let $g \in \C^M$ and $\Lambda \subset \Z_M \times \Z_M$. Consider the Gabor system $(\mathcal{F}_M g, \Lambda')$ with $\Lambda' := \{(\ell, -k) : (k,\ell) \in \Lambda\}$. Then, the frame bounds of $(\mathcal{F}_Mg, \Lambda)$ and $(g, \Lambda')$ coincide. 
\end{proposition}

\begin{proof}
    By direct computation we see that $\mathcal{F}_M M_{\ell} T_k g = e^{2\pi i k\ell/M} M_{- k} T_{\ell} \mathcal{F}_M g$. Therefore, 

    \begin{align*}
        \mathcal{F}_M \Phi_{(g , \Lambda')} \Phi^*_{(g, \Lambda')}\mathcal{F}_M^*(j,j') 
        &= \sum_{(k,l) \in \Lambda} M_{- k} T_{\ell} \mathcal{F}_M g(j) \overline{M_{- k} T_{\ell} \mathcal{F}_M g(j')} \\
        &= \Phi_{(\mathcal{F}_Mg, \Lambda)} \Phi^*_{(\mathcal{F}_Mg, \Lambda)}
    \end{align*}
    Let $v$ be an eigenvector of $\Phi_{(g, \Lambda')} \Phi^*_{(g, \Lambda')}$ with eigenvalue $\lambda$. Then we claim that $\mathcal{F}_M v$ is an eigenvector of $\mathcal{F}_M \Phi_{(g , \Lambda')} \Phi^*_{(g, \Lambda')}\mathcal{F}_M^*(j,j')$ with the same eigenvalue. As the columns of $\mathcal{F}_M$ form an orthonormal basis, we see that $\mathcal{F}_M^*\mathcal{F}_M = I_M$. Then clearly,

    \begin{align*}
        \Phi_{(\mathcal{F}_Mg, \Lambda)} \Phi^*_{(\mathcal{F}_Mg, \Lambda)}\mathcal{F}_M v =  \mathcal{F}_M \Phi_{(g , \Lambda')} \Phi^*_{(g, \Lambda')}\mathcal{F}_M^* \mathcal{F}_M v = \mathcal{F}_M \Phi_{(g , \Lambda')} \Phi^*_{(g, \Lambda')} v = \lambda \mathcal{F}_M v.
    \end{align*}
    This proves that the eigenvalues of $\Phi_{(\mathcal{F}_Mg, \Lambda)} \Phi^*_{(\mathcal{F}_Mg, \Lambda)}$ and $\Phi_{(g, \Lambda')} \Phi^*_{(g, \Lambda')}$ coincide and therefore so do the frame bounds. 
\end{proof}

\section{Estimating frame bounds: general case}\label{sec: frame bounds general}

We start by providing some useful results for estimating optimal frame bounds for general frames. For a frame $\Phi = \{\varphi_j\}_{j = 1}^N\subset \mathbb{C}^M$, let us define a matrix 
\begin{equation}\label{eq: matrix H def}
    H_{\Phi} = \Phi \Phi^* - \frac{N}{M} I_M.
\end{equation}
Following the approach of \cite{fickus2012numerically}, we rely on the following lemma and its corollary.

\begin{lemma}\label{lem: trace}
Let $\Phi = \{\varphi_j\}_{j = 1}^N\subset \mathbb{C}^M$ be a frame and $m \in \N$. Assume that there exists $\delta < 1$ such that for the matrix $H_{\Phi}$ defined in equation~\eqref{eq: matrix H def}
\begin{align*}
    \frac{M^{2m}}{N^{2m}}\Tr\left(H_{\Phi}^{2m}\right) \leq \delta^{2m},
\end{align*}
\noindent Then, for the singular values $\sigma_j(\Phi)$ of the matrix $\Phi$, we have that $\sigma_j^2 \in [(1 - \delta) \frac{N}{M}, (1 + \delta)\frac{N}{M}]$. \linebreak In particular, 
    \begin{align*}
        (1 - \delta) \frac{N}{M} \leq A_\Phi \leq B_\Phi \leq (1 + \delta)\frac{N}{M}.
    \end{align*}
\end{lemma}
\begin{proof}
Let us denote $\delta_{\Phi} := \max_{1 \leq j \leq M} \left|\frac{M} {N}\sigma_j^2 - 1\right|$.
Observe that $H_{\Phi}$ is self-adjoint, and hence diagonizable, with eigenvalues given by $\sigma_j^2 - \frac{N}{M}$. Therefore, the eigenvalues of ${H_{\Phi}^m}$ are given by $(\sigma_j^2 - \frac{N}{M})^m$. Then
\begin{align*}
\delta_{\Phi}^{2m} & = \max_{1 \leq j \leq M} \left|\frac{M} {N}\sigma_j^2 - 1\right|^{2m} \leq \frac{M^{2m}}{N^{2m}}\sum_{j = 1}^M \left|\sigma_j^2 - \frac{N}{M}\right|^{2m} = \frac{M^{2m}}{N^{2m}} \Tr\left(H_{\Phi}^{2m}\right) \leq \delta^{2m}
\end{align*}
\noindent by the proposition assumption, that is, $\delta_{\Phi} \leq \delta < 1$. By the definition of $\delta_\Phi$, we have that ${\sigma_j^2 \in [(1 - \delta(\Phi)) \frac{N}{M}, (1 + \delta (\Phi))\frac{N}{M}]\subset [(1 - \delta) \frac{N}{M}, (1 + \delta)\frac{N}{M}]}$. The bounds on $A_\Phi$ and $B_\Phi$ then follows from
equations~\eqref{eq: lower frame bound def} and~\eqref{eq: upper frame bound def}.
\end{proof}

The following easy corollary relates optimal frame bounds to the frame potential.

\begin{corollary}\label{cr:fpnerf}
    Let $\Phi = \{\varphi_j\}_{j = 1}^N\subset \mathbb{C}^M$ be a unit-norm frame and assume there exists $\delta < 1$ such that
    \begin{align*}
        \frac{M^{2}}{N^{2}}\left(\mathrm{FP}(\Phi) - \frac{N^2}{M}\right) \leq \delta^2.
    \end{align*}
 \noindent Then we have that $\sigma_j^2 \in [(1 - \delta) \frac{N}{M}, (1 + \delta)\frac{N}{M}]$. In particular, 
    \begin{align*}
        (1 - \delta) \frac{N}{M} \leq A_\Phi \leq B_\Phi \leq (1 + \delta)\frac{N}{M}.
    \end{align*}
    \end{corollary}
\begin{proof}
Using the cyclic property of the trace and the fact that $\Phi \Phi^*$ and $\Phi^*\Phi$ are self-adjoint, for the matrix $H_{\Phi}$ defined in equation~\eqref{eq: matrix H def} we have that
        \begin{align*}
            \Tr\left(H_{\Phi}^2\right) 
            &= \Tr\left((\Phi \Phi^*)^2\right) - 2 \frac{N}{M}\Tr(\Phi\Phi^*) + \frac{N^2}{M^2}\Tr (I_M) \\
            &= \Tr\left((\Phi^* \Phi)^2\right) - 2 \frac{N}{M}\Tr(\Phi^*\Phi) + \frac{N^2}{M}\\
            &= \|\Phi^*\Phi\|_{\mathrm{HS}}^2 - 2 \frac{N}{M}\Tr(\Phi^*\Phi) + \frac{N^2}{M}\\
            &= \mathrm{FP}(\Phi) - \frac{N^2}{M}.\numberthis\label{eq: FP via Tr(H)^2}
        \end{align*}
In the last equality here, we used that, since the vectors of the frame $\Phi$ are assumed to have unit norm, $\Tr(\Phi^*\Phi) = N$. The statement then is a direct implication of Lemma~\ref{lem: trace} with $m=1$.
    \end{proof}

The following provides a general lower bound for unit-norm frames via frame-potential. 
\begin{proposition}\label{pr:lowerboundsmallestsing}
Let $\Phi = \{\varphi_j\}_{j = 1}^N\subset\mathbb{C}^M$ be a unit-norm frame. Then
\begin{align*}
    A_\Phi \geq \frac{N}{M} - \frac{M - 1}{2M} - \frac{1}{2} \Tr(H_{\Phi}^2),
\end{align*}
where $H_{\Phi}$ is defined in equation~\eqref{eq: matrix H def}.
\end{proposition}
\begin{proof}
Let $x \in \mathbb{S}^{M - 1}$ and consider  $\Psi := \Phi \cup \{x \}$. Using Proposition~\ref{pr:zerowelch}, we obtain that
\begin{align*}
    \mathrm{FP}(\Phi) + 2 \sum_{j = 1}^N |\langle x , \varphi_j\rangle|^2 + 1 =  \mathrm{FP}(\Psi) \geq \frac{(N + 1)^2}{M}.
\end{align*}
Rearranging and dividing by two yields
\begin{align*}
    \sum_{j = 1}^N |\langle x , \varphi_j\rangle|^2 \geq  \frac{N}{M}  - \frac{M - 1}{2M} + \frac{1}{2}\left(\frac{N^2}{M} - \mathrm{FP}\left(\Phi\right)\right).
\end{align*}
As, by equation~\eqref{eq: FP via Tr(H)^2}, $\mathrm{FP}(\Phi) = \frac{N^2}{M} + \Tr(H_{\Phi}^2)$ we obtain 
\begin{align*}
\sum_{j = 1}^N |\langle x , \varphi_j\rangle|^2 \geq  \frac{N}{M} - \frac{M - 1}{2M} - \frac{1}{2} \Tr(H_{\Phi}^2).
\end{align*}
By taking the minimum over all $x \in S^{M-1}$ it follows

\begin{align*}
    A_\Phi = \min_{x \in S^{M-1}}\|\Phi^*x\|_2^2 \geq \frac{N}{M} - \frac{M - 1}{2M} - \frac{1}{2} \Tr(H_{\Phi}^2),
\end{align*}
and the proof is complete.
\end{proof}

\section{Frame bounds for Gabor subframes} \label{sing_val_sec}

Before we discuss the dependence of the optimal frame bounds on the structure and cardinality of $\Lambda$, let us consider a simple case when set $\Lambda$ has a very particular structure. Namely, we start with the following observation.

\begin{proposition}\label{prop_sing_val_Gabor_regular}
Let $(g, \Lambda)$ be a Gabor system with $\Lambda = F\times \mathbb{Z}_M$  for some $F\subset \mathbb{Z}_M$, $F\ne \emptyset$, and window $g\in \mathbb{C}^M$. Then $(g, \Lambda)$ is a frame if and only if $\min_{m\in \mathbb{Z}_M}\Vert g_{F_m}\Vert_2\ne 0$, where $g_{F_m}$ is the restriction of the vector $g$ to the set of coefficients $F_m = \{m - k\}_{k\in F}\subset \mathbb{Z}_M$. In this case, the optimal lower and upper frame bounds for $(g, \Lambda)$ are given by $$A_{(g, \Lambda)} = M\min_{m\in \mathbb{Z}_M}\Vert g_{F_m}\Vert_2^2,$$ $$B_{(g, \Lambda)} = M\max_{m\in \mathbb{Z}_M} \Vert g_{F_m}\Vert_2^2.$$
\end{proposition}

\begin{proof}
Let us denote by $\Phi_\Lambda\in \mathbb{C}^{M\times |F| M}$ and  the synthesis matrix of the Gabor system $(g, \Lambda)$, and consider its frame operator $\Phi_\Lambda\Phi_\Lambda^*$. For any $m_1, m_2\in \mathbb{Z}_M$,
\begin{equation*}
\begin{split}
\Phi_\Lambda\Phi_\Lambda^*(m_1, m_2) & = \sum_{\lambda\in \Lambda}(\pi(\lambda)g)(m_1)\overline{(\pi(\lambda)g)(m_2)} \\
& = \sum_{k\in F} \sum_{\ell\in \mathbb{Z}_M} e^{2\pi i \ell (m_1 - m_2)/M}g(m_1 - k)\overline{g(m_2 - k)}\\
& = \sum_{k\in F} g(m_1 - k)\overline{g(m_2 - k)} \sum_{\ell\in \mathbb{Z}_M} e^{2\pi i \ell (m_1 - m_2)/M}.
\end{split}
\end{equation*}
\noindent Then, since $\sum_{\ell\in \mathbb{Z}_M} e^{2\pi i \ell (m_1 - m_2)/M} = 0$ for $m_1\ne m_2$, and ${\sum_{\ell\in \mathbb{Z}_M} e^{2\pi i \ell (m_1 - m_2)/M} = M}$ for ${m_1 = m_2}$, we obtain
\begin{equation*}
\Phi_\Lambda\Phi_\Lambda^*(m_1, m_2) = \left\lbrace \begin{array}{ll}
0, & m_1\ne m_2\\
M \sum_{k\in F} |g(m_1 - k)|^2, & m_1 = m_2. 
\end{array}\right.
\end{equation*}
That is, $\Phi_\Lambda\Phi_\Lambda^* = \diag\{M \sum_{k\in F} |g(m - k)|^2\}_{m\in \mathbb{Z}_M}$ is a diagonal matrix and, thus, the set of the singular values of the analysis matrix $\Phi_\Lambda^*$ is given by $\{\sigma_m(\Phi_\Lambda^*)\}_{m\in \mathbb{Z}_M} = \{\sqrt{M} \Vert g_{F_m}\Vert_2\}_{m\in \mathbb{Z}_M}$. Here, $F_m = \{m - k\}_{k\in F}\subset \mathbb{Z}_M$ and $g_S$ denotes the restriction of the vector $g$ to a set of coefficients~$S\subset \mathbb{Z}_M$.

In particular, $(g, \Lambda)$ is a frame if and only if all the diagonal entries of $\Phi_\Lambda\Phi_\Lambda^*$ are nonzero, that is, if and only if $\min_{m\in \mathbb{Z}_M}\Vert g_{F_m}\Vert_2\ne 0$. Moreover, we have
\begin{equation*}
\begin{split}
\sigma_{\min}(\Phi_\Lambda^*) & = \min_{m\in \mathbb{Z}_M}\sigma_m(\Phi_\Lambda^*) = \sqrt{M}\min_{m\in \mathbb{Z}_M}\Vert g_{F_m}\Vert_2;\\
\sigma_{\max}(\Phi_\Lambda^*) & = \max_{m\in \mathbb{Z}_M}\sigma_m(\Phi_\Lambda^*) = \sqrt{M}\max_{m\in \mathbb{Z}_M}\Vert g_{F_m}\Vert_2.
\end{split}
\end{equation*}
That is, $M\min_{m\in \mathbb{Z}_M}\{\Vert g_{F_m}\Vert_2^2\}$ and $M\max_{m\in \mathbb{Z}_M}\{\Vert g_{F_m}\Vert_2^2\}$ are the optimal lower and upper frame bounds for $(g, \Lambda)$, respectively.
\end{proof}

\begin{remark}\label{remark: Fourier duality} We note that, using Proposition~\ref{prop: Fourier duality}, we obtain an analogous result in the the case when the frame set $\Lambda$ is of the form $\Lambda=\mathbb{Z}_M\times F$, for some $F\subset \mathbb{Z}_M$. 
\end{remark}

\medskip

Let us now consider several particular classes of random Gabor windows and use Proposition~\ref{prop_sing_val_Gabor_regular} to estimate the frame bounds for the respective Gabor frames with the frame set of the form ${\Lambda = F\times \mathbb{Z}_M}$.

\begin{example}\label{example_structured}\mbox{}
\begin{enumerate}
\item[(i)] {\it Steinhaus window.} We first consider the case when the window $g$ is chosen so that ${g(m) = \frac{1}{\sqrt{M}}e^{2\pi i y_m}}$, $m\in \mathbb{Z}_M$, and $y_m$ are independent uniformly distributed on~$[0,1)$. Then, for each $m\in \mathbb{Z}_M$, $M \sum_{k\in F} |g(m - k)|^2 = |F|$, and thus $\Phi_\Lambda\Phi_\Lambda^* = |F|I_M$. That is, $A_{(g, \Lambda)} = B_{(g, \Lambda)} = \vert F\vert$, and $(g, \Lambda)$ is a tight frame.

\item[(ii)] {\it Gaussian window.} For a Gaussian window $g\sim \mathcal{C}\mathcal{N}\left( 0, \frac{1}{M}I_M\right)$, we have 
\begin{equation*}
\sigma_m^2(\Phi_\Lambda^*) = M \sum_{k\in F} |g(m - k)|^2 = \sum_{k\in F} \left(\frac{1}{2}2M r(m - k)^2 + \frac{1}{2}2M s(m - k)^2\right),
\end{equation*}
\noindent where $r(m-k) = \mathrm{Re}(g(m - k))$ denotes the real part of $g(m - k)$, and $s(m-k) = \mathrm{Im}(g(m - k))$ denotes its imaginary part. Since, for $k\in F$, $\sqrt{2M}r(m-k),~ \sqrt{2M}s(m-k) \sim \text{i.i.d. } \mathcal{N}(0,1)$ are independent standard Gaussian random variables, we can apply Lemma \ref{chi_square} to obtain that, for any $t>0$,
\begin{equation*}
\begin{split}
& \mathbb{P}\left\lbrace \sigma_m^2(\Phi_\Lambda^*) \ge |F| + \sqrt{2|F|t} + t \right\rbrace \le e^{-t};\\
& \mathbb{P}\left\lbrace \sigma_m^2(\Phi_\Lambda^*) \le |F| - \sqrt{2|F|t} \right\rbrace \le e^{-t}.
\end{split}
\end{equation*}
\noindent Then, setting $t =2 |F|$ in the first equation and $t = \frac{1}{8}|F|$ in the second one, we obtain
\begin{equation*}
\begin{split}
& \mathbb{P}\left\lbrace \sigma_m^2(\Phi_\Lambda^*) \ge 5|F| \right\rbrace \le e^{-2|F|};\\
& \mathbb{P}\left\lbrace \sigma_m^2(\Phi_\Lambda^*) \le \frac{1}{2}|F| \right\rbrace \le e^{-\frac{|F|}{8}}.
\end{split}
\end{equation*}
Suppose now that $|F|\ge C\log M$, for some sufficiently large constant $C>0$. Then, combining the probability estimates obtained above and taking the union bound over all $m\in \mathbb{Z}_M$, we obtain that, with high probability,
\begin{equation*}
\frac{1}{2}|F|< \sigma_m^2(\Phi_\Lambda^*)< 5|F|,
\end{equation*}
\noindent for all $m\in \mathbb{Z}_M$. In particular, for the frame bounds of $(g, \Lambda)$ we have 
\begin{equation}\label{example_gaussian_wind}
\frac{1}{2}|F|< A_{(g, \Lambda)} \le B_{(g, \Lambda)}< 5|F|.
\end{equation}

\item[(iii)] {\it Window, uniformly distributed on $\mathbb{S}^{M-1}$.} It is a well-known fact that a window $g$, uniformly distributed on the unit sphere $\mathbb{S}^{M-1}$, can be written in the form $g = h/\Vert h\Vert_2$, where $h\sim \mathcal{C}\mathcal{N}\left(0,\frac{1}{M} I_M\right)$ \cite{marsaglia1972choosing}. Moreover, Lemma \ref{lemma_norm_gaussian} shows that, for some $C>0$, $\frac{1}{2}\le \Vert h\Vert_2\le 2$ with probability at least $1 - e^{-CM}$. Thus, with the same probability, 
\begin{equation*}
\frac{1}{4} M \sum_{k\in F} |h(m - k)|^2 \le M \sum_{k\in F} |g(m - k)|^2 \le  4M \sum_{k\in F} |h(m - k)|^2.
\end{equation*}
\noindent Combining this with \eqref{example_gaussian_wind}, we obtain that with high probability
\begin{equation*}
\frac{1}{8}|F|< A_{(g, \Lambda)} \le B_{(g, \Lambda)}< 20|F|.
\end{equation*}
\end{enumerate}
\end{example}

\medskip
The examples above show that, in the case when $\Lambda$ has a regular structure and window $g$ is random, the Gabor frame $(g,\Lambda)$ has frame bounds that are quite close to each other, and, thus, is well-conditioned.  

\begin{remark}
    Let $\mathcal{C} \subset \mathcal{P}(\Z_M \times \Z_M)$ be a collection of subsets of $\Z_M \times \Z_M$. Suppose that we know (or have a good estimate of) the optimal frame bounds for all the Gabor frames $(g,\Lambda')$ with $\Lambda' \in \mathcal{C}$. This can be used to derive estimates the lower and upper frame bounds for $(g,\Lambda)$, such that there exist $\Lambda', \Lambda''\in \mathcal{C}$ with $\Lambda'\subset \Lambda \subset \Lambda''$. Indeed, if $\Lambda' \subset \Lambda$, then we have
\begin{align*}
    \min_{x \in S^{M-1}} \sum_{\lambda \in \Lambda'} \left\vert \langle x , \pi(\lambda)g \rangle\right\vert^2 \leq \min_{x \in S^{M-1}}\sum_{\lambda \in \Lambda} \left\vert \langle x , \pi(\lambda)g\rangle \right\vert^2.
\end{align*}
Thus the smallest singular value of $(g,\Lambda)$ is at least as large as the smallest singular value of $(g,\Lambda')$ (which we assumed to be known). By similar reasoning, if $\Lambda \subset \Lambda''$ then the largest singular value of $(g,\Lambda)$ is at most as large as $(g,\Lambda'')$. 

By setting $\mathcal{C} = \{F \times \Z_M \colon \varnothing \subsetneq F \subset \Z_M\} \cup \{\Z_M \times F\colon \varnothing \subsetneq F \subset \Z_M\}$, this argument can be used to extend the results of Proposition~\ref{prop_sing_val_Gabor_regular} and Remark~\ref{remark: Fourier duality} to a larger class of Gabor frames $(g,\Lambda)$ with $\Lambda'\subset \Lambda \subset \Lambda''$ for some $\Lambda', \Lambda''\in \mathcal{C}$.
\end{remark}

\subsection{The case of a general $\Lambda$}\label{sec_proof_of_main_results}

In this section, we consider the case when $\Lambda$ is an arbitrary subset of $\mathbb{Z}_M\times\mathbb{Z}_M$ and derive frame bounds for Gabor frames with Steinhaus and Gaussian windows. As we mentioned before, the frame bounds of a Gabor frame depend on the structure of the frame set $\Lambda$. The result below evaluates the frame bounds independently of its structure, depending only on the cardinality of $\Lambda$. Thus, one should consider this result as the worst case bound (compare, for instance, to the bounds established in Example~\ref{example_structured} for sets $\Lambda$ with specific structure). 

We start our consideration by showing the following technical lemma, which follows the idea of \cite[Lemma~3.4]{pfander2010sparsity}.

\begin{lemma}\label{lemma_trace_formula}
Consider a Gabor system $(g, \Lambda)$ with $\Lambda\subset \mathbb{Z}_M\times \mathbb{Z}_M$ and a random window $g$. Then, for any $m\in \mathbb{N}$ and $\delta >0$,
\begin{equation*}
\mathbb{P}\left\lbrace\frac{|\Lambda|}{M}(1-\delta) \le A_{(g, \Lambda)} \le B_{(g, \Lambda)}\le \frac{|\Lambda|}{M}(1+\delta)\right\rbrace \geq 1 - \frac{M^{2m}}{|\Lambda|^{2m}}\delta^{-2m}\mathbb{E}(\Tr H^{2m}),
\end{equation*}
\noindent where $H = \Phi_\Lambda \Phi_\Lambda ^*- \frac{|\Lambda|}{M}I_M$ is as in~\eqref{eq: matrix H def}. Furthermore, if $g$ is a Steinhaus window, for any $m\in \mathbb{N}$,
\begin{equation*}
\mathbb{E}\left(\Tr H^m \right) = \sum_{\substack{j_1, j_2, \dots, j_m\in \mathbb{Z}_M,\\j_1\ne j_2\ne\dots\ne j_m\ne j_{1}}}\sum_{(k_1, \ell_1)\in \Lambda}  \dots \sum_{(k_m, \ell_m)\in \Lambda} e^{\frac{2\pi i}{M} \sum_{t = 1}^m \ell_t (j_t - j_{t+1})} E_{\substack{j_1\dots j_m\\k_1\dots k_m}},
\end{equation*}
\noindent where $E_{\substack{j_1\dots j_m\\k_1\dots k_m}} =\frac{1}{M^{m}}$, if there exists a bijection $\alpha:\{1,\dots,m\}\to\{1,\dots,m\}$, such that ${j_t - k_t = j_{\alpha(t)} - k_{\alpha(t) - 1}}$, for all $t\in\{1,\dots,m\}$; and $E_{\substack{j_1\dots j_m\\k_1\dots k_m}} = 0$, otherwise.
\end{lemma}

\begin{proof}
First, we note that 
\begin{equation*}
\mathbb{P}\left\lbrace\frac{|\Lambda|}{M}(1-\delta) \le A_{(g,\Lambda)} \le B_{(g,\Lambda)}\le \frac{|\Lambda|}{M}(1+\delta)\right\rbrace  = \mathbb{P}\left\lbrace \Vert H\Vert_2 \le \frac{|\Lambda|}{M}\delta\right\rbrace.
\end{equation*}

Using Markov’s inequality, the fact that the Frobenius norm majorizes the operator norm, and the fact that $H$ is self-adjoint, for any $m\in \mathbb{N}$ we have
\begin{align*}
\mathbb{P} \left\lbrace\Vert H\Vert_2> \frac{|\Lambda|}{M}\delta\right\rbrace & = \mathbb{P} \left\lbrace\Vert H\Vert _2^{2m}> \frac{|\Lambda|^{2m}}{M^{2m}}\delta^{2m}\right\rbrace \le \frac{M^{2m}}{|\Lambda|^{2m}}\delta^{-2m}\mathbb{E}(\Vert H\Vert _2^{2m})  \\
& = \frac{M^{2m}}{|\Lambda|^{2m}}\delta^{-2m}\mathbb{E}(\Vert H^m\Vert _2^{2}) \le \frac{M^{2m}}{|\Lambda|^{2m}}\delta^{-2m}\mathbb{E}(\Vert H^m\Vert _F^{2}) \\
& = \frac{M^{2m}}{|\Lambda|^{2m}}\delta^{-2m}\mathbb{E}(\Tr H^{2m}),
\end{align*}
\noindent which concludes the proof of the first part of the lemma. 

For the second part, we need to estimate the trace expectation $\mathbb{E}(\Tr H^{2m})$. For any $j_1, j_2\in \mathbb{Z}_M$, 
\begin{equation*}
\Phi_\Lambda \Phi_\Lambda^* (j_1, j_2) = \sum_{(k, \ell)\in \Lambda}e^{2\pi i \ell (j_1 - j_2)/M}g(j_1 - k)\overline{g(j_2 - k)}.
\end{equation*}
Thus, since we assume $g$ to be Steinhaus, $|g(j)| = \frac{1}{\sqrt{M}}$, for all $j\in \mathbb{Z}_M$. For $H$ we have
\begin{equation*}
H (j_1, j_2) = \left\lbrace \begin{array}{ll}
\sum_{(k, \ell)\in \Lambda}e^{2\pi i \ell (j_1 - j_2)/M}g(j_1 - k)\overline{g(j_2 - k)}, & j_1\ne j_2;\\
0, & j_1 = j_2.
\end{array}\right.
\end{equation*}
Then, for $j_1,\dots, j_{m+1}\in \mathbb{Z}_M$, we recursively obtain
\begin{equation*}
\Scale[0.93]{
\begin{split}
& H^2 (j_1, j_3) = \sum_{j_2\in \mathbb{Z}_M} H(j_1,j_2)H(j_2, j_3) \\
& = \sum_{\substack{j_2\in \mathbb{Z}_M,\\j_2\ne j_1,j_3}}\sum_{(k_1, \ell_1)\in \Lambda} \sum_{(k_2, \ell_2)\in \Lambda} e^{\frac{2\pi i}{M} (\ell_1 (j_1 - j_2) + \ell_2 (j_2 - j_3))} g(j_1 - k_1)\overline{g(j_2 - k_1)}g(j_2 - k_2)\overline{g(j_3 - k_2)};\\
& H^3 (j_1, j_4) = \sum_{j_3\in \mathbb{Z}_M} H^2(j_1,j_3)H(j_3, j_4) \\
& = \sum_{\substack{j_3\in \mathbb{Z}_M,\\j_3\ne j_4}}\sum_{\substack{j_2\in \mathbb{Z}_M,\\j_2\ne j_1,j_3}}\sum_{(k_1, \ell_1)\in \Lambda} \sum_{(k_2, \ell_2)\in \Lambda} \sum_{(k_3, \ell_3)\in \Lambda} e^{\frac{2\pi i}{M} \sum_{t = 1}^3 \ell_t (j_t - j_{t+1})} \prod_{t = 1}^3 g(j_t - k_t)\overline{g(j_{t+1} - k_t)};
\end{split}}
\end{equation*}
and, in general,
\begin{equation*}
\Scale[0.93]{
\begin{split}
& H^m (j_1, j_{m+1}) = \sum_{j_m\in \mathbb{Z}_M} H^{m-1}(j_1,j_m)H(j_m, j_{m+1}) \\
& = \sum_{\substack{j_m\in \mathbb{Z}_M,\\j_m\ne j_{m+1}}}\dots\sum_{\substack{j_3\in \mathbb{Z}_M,\\j_3\ne j_4}}\sum_{\substack{j_2\in \mathbb{Z}_M,\\j_2\ne j_1,j_3}}\sum_{(k_1, \ell_1)\in \Lambda}  \dots \sum_{(k_m, \ell_m)\in \Lambda} e^{\frac{2\pi i}{M} \sum_{t = 1}^m \ell_t (j_t - j_{t+1})} \prod_{t = 1}^m g(j_t - k_t)\overline{g(j_{t+1} - k_t)}.
\end{split}}
\end{equation*}
Thus, for the trace of the matrix $H^m$, we have
\begin{equation*}
\Scale[0.93]{
\begin{split}
& \Tr(H^m) = \sum_{\substack{j_1, j_2, \dots, j_m\in \mathbb{Z}_M,\\j_1\ne j_2\ne\dots\ne j_m\ne j_{1}}}\sum_{(k_1, \ell_1)\in \Lambda}  \dots \sum_{(k_m, \ell_m)\in \Lambda} e^{\frac{2\pi i}{M} \sum_{t = 1}^m \ell_t (j_t - j_{t+1})} \prod_{t = 1}^m g(j_t - k_t)\overline{g(j_{t+1} - k_t)}, \text{ and}\\
& \mathbb{E}\left(\Tr(H^m)\right) = \sum_{\substack{j_1, j_2, \dots, j_m\in \mathbb{Z}_M,\\j_1\ne j_2\ne\dots\ne j_m\ne j_{1}}}\sum_{(k_1, \ell_1)\in \Lambda}  \dots \sum_{(k_m, \ell_m)\in \Lambda} e^{\frac{2\pi i}{M} \sum_{t = 1}^m \ell_t (j_t - j_{t+1})} E_{\substack{j_1\dots j_m\\k_1\dots k_m}},
\end{split}}
\end{equation*}
\noindent where $ E_{\substack{j_1\dots j_m\\k_1\dots k_m}} = \mathbb{E}\left(\prod_{t = 1}^m g(j_t - k_t)\overline{g(j_{t+1} - k_t)}\right)$.

Let us compute $E_{\substack{j_1\dots j_m\\k_1\dots k_m}}$ now. Since $g(j)$, $j\in \mathbb{Z}_M$, are independent, the expectation can be factored into a product of the form
\begin{equation*}
\mathbb{E}\left(\prod_{t = 1}^m g(j_t - k_t)\overline{g(j_{t+1} - k_t)}\right) = \prod_{j\in \mathbb{Z}_M}\mathbb{E}\left( g(j)^{\mu_j}\overline{g(j)}^{\nu_j} \right),
\end{equation*}
\noindent for some $\mu_j, \nu_j\in \mathbb{N}\cup \{0\}$. Moreover, since $\sqrt{M}g(j)$ is uniformly distributed on the unit torus $\{z\in \mathbb{C}: \Vert z\Vert _2 = 1\}$ and $\mathbb{E}\left( g(j) \right) = 0$, we have
\begin{equation*}
\mathbb{E}\left( g(j)^{\mu_j}\overline{g(j)}^{\nu_j} \right) = \left\lbrace \begin{array}{ll}
\mathbb{E}\left( |g(j)|^{2\mu_j} \right) = \frac{1}{M^{\mu_j}}, & \mu_j = \nu_j;\\
0, & \mu_j \ne \nu_j.
\end{array}\right.
\end{equation*}
Thus, under the convention that $k_{0} = k_{m}$,
\begin{equation}\label{eq: exptession for E}
E_{\substack{j_1\dots j_m\\k_1\dots k_m}} = \left\lbrace \begin{array}{ll}
\frac{1}{M^{m}}, & \text{if} ~ \exists \text{ bijection } \alpha: \{1,\dots,m\} \to \{1,\dots,m\}, \\
&\text{s.t. }\forall t\in \{1,\dots,m\} ~  j_t - k_t = j_{\alpha(t)} - k_{\alpha(t) - 1}; \\
0, & \text{otherwise.}
\end{array}\right.
\end{equation}
This concludes the proof.
\end{proof}

The following proposition gives a direct computation for the expected trace for $m = 1$.

\begin{proposition}[Trace Steinhaus]\label{pr:tracesteinhaus}
 Let $g$ be a Steinhaus window and $\Lambda \subset \Z_M \times \Z_M$. Then, 
    \begin{align*}
        \mathbb{E}\left(\Tr H^2 \right) = |\Lambda| - \frac{1}{M}\sum_{k \in \Z_M} |A_k|^2,
    \end{align*}
    where $A_k = \{\ell : (k,\ell) \in \Lambda\}$.
\end{proposition}
\begin{proof}
    Following the proof of Lemma~\ref{lemma_trace_formula}, we obtain that 
    \begin{align*}
        \mathbb{E}\left(\Tr \ H^2 \right) = \sum_{j \in \Z_M} \sum_{j' \in \Z_M \atop j'\neq j} \sum_{(k,\ell) \in \Lambda} \sum_{(k',\ell') \in \Lambda} e^{2\pi i (\ell - \ell') (j - j')/M} E_{j, j' \atop k , k'},
    \end{align*}
    where $E_{j, j' \atop k , k'} := \mathbb{E}\left(g(j - k) \overline{g(j' -k)} g(j' - k') \overline{g(j- k')}\right)$. It follows from~\eqref{eq: exptession for E}, $E_{j,j '\atop k, k'}$ is non-zero if and only if one of the following systems of equations holds
    \begin{align*}
        \begin{cases}
            j - k = j - k' \\
            j' - k = j' - k'
        \end{cases}
        \qquad \qquad
        \begin{cases}
            j - k = j' - k \\
            j' - k' = j' - k
        \end{cases}
    \end{align*}
    It is clear that the first system has a solution if and only if $k = k'$. The second system of equations does not have a solution, as in the sum computing $\mathbb{E}(\Tr H^2)$, we have $j \neq j'$. Therefore, $$E_{j, j' \atop k, k'} = \begin{cases}
        \frac{1}{M^2}, & k = k',\\
        0, & k\ne k'.
    \end{cases}$$

    Let us define $A_k := \{\ell : (k,\ell) \in \Lambda\}$. Clearly, $\sum_{k \in \Z_M} |A_k| = |\Lambda|$. Using the previous observations, we compute the trace expectation for the Steinhaus window as 
    \begin{equation*}
    \begin{split}
    \mathbb{E}(\mathrm{Tr} \ H^2) &= \frac{1}{M^2} \sum_{j, j' \in \Z_M \atop j \neq j'} \sum_{k \in \Z_M} \sum_{\ell \in A_k} \sum_{\ell' \in A_k} e^{\frac{2 \pi i}{M} (\ell - \ell')(j - j')} \\
    &= \frac{1}{M^2}\sum_{j \in \Z_M} \sum_{k \in \Z_M} \sum_{\ell \in A_k} \left(\sum_{\ell' \in A_k \atop \ell \neq \ell'} \sum_{j' \in \Z_M \atop j' \neq j}e^{\frac{2 \pi i}{M} (\ell - \ell')(j - j') }+ \sum_{j' \in \Z_M \atop j' \neq j} 1\right) \\
    &= \frac{1}{M^2 }\sum_{j \in \Z_M} \sum_{k \in \Z_M} \sum_{\ell \in A_k} \left(\sum_{\ell' \in A_k \atop \ell \neq \ell'} -1 + (M - 1)\right) \\
    &\leq \frac{1}{M }\sum_{k \in \Z_M} \sum_{\ell \in A_k} \left(-(|A_k| - 1) + (M - 1)\right) = \frac{1}{M }\sum_{k \in \Z_M} |A_k| \left(M - |A_k|\right) \\
    &= \sum_{k \in \Z_M} |A_k| - \frac{1}{M} |A_k|^2 = |\Lambda| - \frac{1}{M} \sum_{k \in \Z_M} |A_k|^2 
    \end{split}
\end{equation*}
This completes the proof. 
    
\end{proof}

Using this expected trace computation together with Lemma~\ref{lemma_trace_formula}, we obtain the following estimations for the optimal frame bounds of Gabor frames.

\begin{corollary}\label{cor: m2 Steinhaus general}
    Let $g\in \mathbb{C}^M$ be a Steinhaus window and consider a Gabor system $(g, \Lambda)$ with $\Lambda\subset \mathbb{Z}_M\times \mathbb{Z}_M$. Then
    \begin{enumerate}
        \item For any $\varepsilon\in (0,1)$,
\begin{equation*}
\mathbb{P}\left(B_{(g, \Lambda)}\le \frac{|\Lambda|}{M} + \sqrt{\frac{|\Lambda|}{\varepsilon}\left(1 - \frac{|\Lambda|}{M^2}\right)}\right) \ge 1 - \varepsilon.
\end{equation*}

\item  Let $\varnothing \subsetneq F \subset \Z_M$ with $\#F = \alpha M$, and $\delta < 1$ be arbitrary. Assume that $\Lambda \subset F \times \Z_M$ with $|\Lambda| = (1 - p)\alpha M^2$, then
    \begin{align*}
        \mathbb{P}\left(\frac{|\Lambda|}{M}(1-\delta) \le A_{(g, \Lambda)} \le B_{(g, \Lambda)}\le \frac{|\Lambda|}{M}(1+\delta) \right) \ge 1 - \frac{p}{\alpha(1 - p)} \frac{1}{\delta^2}.
    \end{align*}
    \end{enumerate}
\end{corollary}

\begin{proof}

To prove the first statement, note that by Cauchy-Schwarz inequality, $\sum_{k\in \mathbb{Z}_M} |A_k|^2 \ge \frac{1}{M} \left( \sum_{k\in \mathbb{Z}_M} |A_k| \right)^2$, and thus 
\begin{align*}
        |\Lambda| - \frac{1}{M} \sum_{k \in F} |A_k|^2 \leq |\Lambda| \left(1 - \frac{|\Lambda|}{M^2} \right).
    \end{align*}
Then, setting $\delta = \sqrt{\frac{M^2 - |\Lambda|}{\varepsilon |\Lambda|}}$ for some $\varepsilon\in (0,1)$ and using Lemma~\ref{lemma_trace_formula}, we obtain the desired inequality with probability at least $1 - \varepsilon$.

\medskip

To prove the second part of the corollary, note that 
    \begin{align*}
        |\Lambda| - \frac{1}{M} \sum_{k \in F} |A_k|^2 \leq |\Lambda| - \frac{1}{M \#F} |\Lambda|^2 \leq |\Lambda| - \frac{1}{\alpha M^2} |\Lambda|^2.
    \end{align*}
Hence, using $|\Lambda| = (1 - p)\alpha M^2$,  

    \begin{align*}
         \frac{M^2}{|\Lambda|^2} \mathbb{E} \left(\mathrm{Tr} \ H^2\right) \leq \frac{M^2}{|\Lambda|^2} \left(|\Lambda| - \frac{1}{\alpha M^2} |\Lambda|^2\right) \leq \frac{p}{\alpha(1 - p)}. 
    \end{align*}
The result follows by Lemma~\ref{lemma_trace_formula}.
    
\end{proof}

We note that the bound obtained in Corollary \ref{cor: m2 Steinhaus general} is tight for a full Gabor frame, \linebreak when~${\Lambda = \mathbb{Z}_M\times \mathbb{Z}_M}$. In the case when $|\Lambda| = \alpha M^2$, for some $\alpha\in (0,1)$, the proven bound gives $B_{(g, \Lambda)}\le \left(\alpha + \sqrt{\frac{\alpha(1 - \alpha)}{\varepsilon}}\right)M = \left(1 + \sqrt{\frac{(1 - \alpha)}{\alpha\varepsilon}}\right)\frac{|\Lambda|}{M}$ with probability at least $1 - \varepsilon$. That is, the bound on the upper frame bound $B_{(g, \Lambda)}$ in this case is the same (up to a constant), as the one obtained in~\cite{latala2005some} for random frames with frame vectors whose entries are independent identically distributed random variables with bounded fourth moment.

\medskip

The trace evaluation method established in Lemma~\ref{lemma_trace_formula} can be applied for other random windows, such as Gaussian windows. For the trace expectation in this case, we have the following result.

\begin{proposition}[Trace Gaussian]\label{pr:trgauss}
Let $g$ be a Gaussian window and $\Lambda \subset \Z_M \times \Z_M$. Then,
\begin{align}
    \mathbb{E}\left(\Tr H^2\right) = |\Lambda|.
\end{align}
\end{proposition}
\begin{proof}
    The proof of this statement is mostly analogous to the computation in case of a Steinhaus window. The key observation to make is that $H$ no longer has zeros on the diagonal. However, by similar computation we obtain
    \begin{align*}
        H^2(j,j) &=  H(j,j)^2 +  \sum_{j' \in \Z_m \atop j' \neq j} H(j,j') H(j',j)\\
        &=  H(j,j)^2 + \sum_{j'\in \Z_M \atop j \neq j'} \sum_{(k,\ell) \in \Lambda} \sum_{(k',\ell') \in \Lambda} e^{2\pi i (\ell - \ell') (j - j')/M} g(j - k) \overline{g(j' -k)} g(j' - k') \overline{g(j- k')}.
    \end{align*}
    Therefore,
    \begin{align*}
        \mathbb{E}(\Tr H^2)
        &=  \sum_{j \in \Z_M} \mathbb{E}(H(j,j)^2) + \sum_{j\in \Z_M}\sum_{j'\in \Z_M \atop j \neq j'} \sum_{(k,\ell) \in \Lambda} \sum_{(k',\ell') \in \Lambda} e^{2\pi i (\ell - \ell') (j - j')/M} E_{j,j' \atop k,k'}.
    \end{align*}
    Let us observe that, since $j' \ne j$, and using independence of the entries of $g$, 
    $$E_{j,j' \atop k, k'}  = \begin{cases}
        0, & k \ne k;\\
        \frac{1}{M^2}, & k = k. 
    \end{cases}$$
    Thus, 
    \begin{align*}
       & \sum_{j\in \Z_M}\sum_{j'\in \Z_M \atop j \neq j'} \sum_{(k,\ell) \in \Lambda} \sum_{(k',\ell') \in \Lambda} e^{2\pi i (\ell - \ell') (j - j')/M} E_{j,j' \atop k,k'} = \frac{1}{M^2} \sum_{j\in \Z_M}\sum_{j'\in \Z_M \atop j \neq j'} \sum_{(k,\ell) \in \Lambda}\left( 1 + \sum_{\ell'\in A_k \atop \ell'\ne \ell} e^{2\pi i (\ell - \ell') (j - j')/M}\right)\\
       & = \frac{\vert \Lambda \vert M(M-1)}{M^2} + \frac{1}{M^2} \sum_{j\in \Z_M} \sum_{(k,\ell) \in \Lambda} \sum_{\ell'\in A_k \atop \ell'\ne \ell} \sum_{j'\in \Z_M \atop j \neq j'} e^{2\pi i (\ell - \ell') (j - j')/M} = \frac{\vert \Lambda \vert (M-1)}{M} - \frac{1}{M} \sum_{(k,\ell) \in \Lambda} \sum_{\ell'\in A_k \atop \ell'\ne \ell} 1\\
       & = \vert \Lambda \vert - \frac{\vert \Lambda \vert}{M} - \frac{1}{M} \sum_{k\in \mathbb{Z}_M} \vert A_k\vert (\vert A_k \vert - 1)  = \vert \Lambda \vert - \frac{1}{M}\sum_{k\in \mathbb{Z}_M} \vert A_k\vert^2.
    \end{align*}
    
    It remains to calculate $\sum_{j \in \Z_M} \mathbb{E}(H(j,j)^2)$. As $g$ is a Gaussian window, we have that ${g(j) = a(j) + b(j) i}$. It is easily verified that $\mathbb{E}(|g(j)|^2) = \frac{1}{M}$ and $\mathbb{E}(|g(j)|^4) = \frac{2}{M^2}$. Observe that $H(j,j) = - \frac{|\Lambda|}{M} +  \sum_{(k,l) \in \Lambda} |g(j - k)|^2$ and hence
    \begin{align*}
        H(j,j)^2 &= \left(-\frac{|\Lambda|}{M} + \sum_{(k,l) \in \Lambda} |g(j - k)|^2\right)^2 \\ 
    & = \frac{|\Lambda|^2}{M^2} - 2 \frac{|\Lambda|}{M} \sum_{(k,l) \in \Lambda} |g(j - k)|^2 + \sum_{(k,\ell) \in \Lambda} \sum_{(k',\ell') \in \Lambda} |g(j - k)|^2 |g(j - k')|^2.
    \end{align*}
    It follows that,
    \begin{align*}
    \mathrm{E}(H(j,j)^2) = - \frac{|\Lambda|^2}{M^2}+ \mathbb{E}\left(\sum_{(k,\ell) \in \Lambda} \sum_{(k',\ell') \in \Lambda} |g(j - k)|^2 |g(j - k')|^2\right).
    \end{align*}

    Observe that if $k \neq k'$, then $\mathbb{E}(|g(j - k)|^2 |g(j - k')|^2) = \frac{1}{M^2}$ and otherwise $\mathbb{E}(|g(j - k)|^2 |g(j - k')|^2) = \mathbb{E}(|g(j - k)|^4) = \frac{2}{M^2}$. It is clear that there are precisely $\sum_{k \in \Z_M} |A_k|^2$ occurrences where $k = k'$. Therefore, 
    \begin{align*}
        \mathbb{E}\left(\sum_{(k,\ell) \in \Lambda} \sum_{(k',\ell') \in \Lambda} |g(j - k)|^2 |g(j - k')|^2\right) &= \frac{2}{M^2}\sum_{k \in \Z_M} |A_k|^2 + \left(|\Lambda|^2 - \sum_{k \in \Z_M} |A_k|^2\right) \frac{1}{M^2} \\
        &= \frac{|\Lambda|^2}{M^2} + \frac{1}{M^2} \sum_{k \in \Z_M} |A_k|^2
    \end{align*}

    By putting everything together, we then obtain that
    \begin{align*}
        \mathbb{E}(\Tr \ H^2) = \left(|\Lambda| - \frac{1}{M} \sum_{k \in \Z_M} |A_k|^2\right) + \frac{1}{M} \sum_{k \in \Z_M} |A_k|^2 = |\Lambda|.
    \end{align*}
    This concludes the proof.    
\end{proof}

We can now use this to evaluate the upper frame bound for a Gabor frame with a Gaussian window.

\begin{corollary}\label{cor:gausstrace}
Let $g$ be a Gaussian window and $\Lambda \subset \Z_M \times \Z_M$ of size $\vert \Lambda \vert = (1 - p)M^2$, for some $p\in (0,1)$. Then, for any $\varepsilon \in (0,1)$, 
\begin{align*}
\mathbb{P}\left(B_{(g,\Lambda)} \leq M\left(1 - p + \sqrt{\frac{1 - p}{\varepsilon}}\right)\right) \ge 1 - \varepsilon.
\end{align*}
\end{corollary}
\begin{proof}
Lemma~\ref{lemma_trace_formula} and Proposition~\ref{pr:trgauss} imply that
\begin{align*}
\mathbb{P}\left(B_{(g,\Lambda)} > (1 + \delta)\frac{|\Lambda|}{M} \right) \le \frac{M^{2}}{|\Lambda|^{2}}\frac{1}{\delta^{2}}\mathbb{E}\left(\Tr H^{2}\right) = \frac{1}{1 - p} \frac{1}{\delta^2}.
\end{align*}
The proof is concluded by setting $\delta = \frac{1}{\sqrt{\varepsilon(1 - p)}}$. 
\end{proof}

\subsection{Gabor subframes with a random $\Lambda$}\label{sec:randomlambda}

Let us now consider the case of a Gabor frame with a randomly selected frame set $\Lambda$. Roughly speaking, the result below shows that, for any $\epsilon\in (0,1)$, \emph{most} of the subframes $(g, \Lambda)$ of the full Gabor frame $(g, \mathbb{Z}_M\times \mathbb{Z}_M)$ with $|\Lambda| = O(M^{1+\epsilon})$ are well-conditioned.

\begin{theorem}\label{th_sing_val_rand_lambda}
Let $g\in \mathbb{C}^M$ be a Steinhaus window. For any even $m\in \mathbb{N}$, consider a Gabor system $(g, \Lambda)$ with a random set $\Lambda\subset \mathbb{Z}_M\times \mathbb{Z}_M$ constructed so that the events $\{(k,\ell)\in \Lambda\}$ are independent for all $(k,\ell)\in \mathbb{Z}_M\times \mathbb{Z}_M$ and have probability $\tau = \frac{C\log M}{M^{\frac{m-1}{m}}}$, for a sufficiently large constant $C>0$ depending only on $m$. Then, for any $\delta >0$, 
\begin{equation*}
\mathbb{P}\left\lbrace\frac{|\Lambda|}{M}(1-\delta) \le A_{(g, \Lambda)} \le B_{(g, \Lambda)}\le \frac{|\Lambda|}{M}(1+\delta)\right\rbrace\ge 1 - \varepsilon,
\end{equation*}
\noindent where $\varepsilon\in (0,1)$ depends on $m$, $\delta$, and the choice of $C$.
\end{theorem}

\begin{proof}
For a realization of $\Lambda$, Lemma~\ref{lemma_trace_formula} implies
\begin{equation*}
\begin{gathered}
\mathbb{P}_g\left\lbrace\frac{|\Lambda|}{M}(1-\delta) \le A_{(g, \Lambda)} \le B_{(g, \Lambda)}\le \frac{|\Lambda|}{M}(1+\delta)\right\rbrace \geq 1 - \frac{M^{m}}{|\Lambda|^{m}}\delta^{-m}\mathbb{E}_g(\Tr H^{m}),\\
\mathbb{E}_g\left(\Tr H^m \right) = \sum_{\substack{j_1, j_2, \dots, j_m\in \mathbb{Z}_M,\\j_1\ne j_2\ne\dots\ne j_m\ne j_{1}}}\sum_{(k_1, \ell_1)\in \Lambda}  \dots \sum_{(k_m, \ell_m)\in \Lambda} e^{\frac{2\pi i}{M} \sum_{t = 1}^m \ell_t (j_t - j_{t+1})} E_{\substack{j_1\dots j_m\\k_1\dots k_m}},
\end{gathered}
\end{equation*}
\noindent where $E_{\substack{j_1\dots j_m\\k_1\dots k_m}} =\frac{1}{M^{m}}$ if there exists a permutation $\alpha\in \Sigma_m$, such that, for  every $t\in \{1,\dots,m\}$, $j_t - k_t = j_{\alpha(t)} - k_{\alpha(t) - 1}$; and $E_{\substack{j_1\dots j_m\\k_1\dots k_m}} = 0$ otherwise.

As before, let us denote $A_k = \{\ell\in \mathbb{Z}_M, \text{ s.t. }(k,\ell)\in \Lambda\}$. After rearranging the sum in the trace formula above, we have
\begin{equation}\label{eq: trace expression}
\begin{split}
\mathbb{E}\left(\Tr H^m \right) & = \sum_{\substack{j_1, j_2, \dots, j_m\in \mathbb{Z}_M,\\j_1\ne j_2\ne\dots\ne j_m\ne j_{1}}}\sum_{k_1, k_2, \dots, k_m\in \mathbb{Z}_M} E_{\substack{j_1\dots j_m\\k_1\dots k_m}} \sum_{\ell_1\in A_{k_1}}  \dots \sum_{\ell_m\in A_{k_m}} e^{\frac{2\pi i}{M} \sum_{t = 1}^m \ell_t (j_t - j_{t+1})}\\
& = \sum_{\substack{j_1, j_2, \dots, j_m\in \mathbb{Z}_M,\\j_1\ne j_2\ne\dots\ne j_m\ne j_{1}}}\sum_{k_1, k_2, \dots, k_m\in \mathbb{Z}_M} E_{\substack{j_1\dots j_m\\k_1\dots k_m}} \prod_{t = 1}^m \sum_{\ell_t\in A_{k_t}} e^{\frac{2\pi i}{M}\ell_t (j_t - j_{t+1})}
\end{split}
\end{equation}
We note that, by the construction of $\Lambda$, each set $A_{k_t}$, $t\in \{1,\dots, m\}$, is a random subset of $\mathbb{Z}_M$, such that the events $\{\ell\in A_{k_t}\}$, $\ell\in \mathbb{Z}_M$, are independent and have probability $\tau$. Then Corollary \ref{cor_sum_rand_roots_of_unity} implies that, for every $t\in \{1,\dots, m\}$ and a constant $C' > 4\sqrt{2}$,
\begin{equation*}
\mathbb{P}\left\lbrace \max_{q\in \mathbb{Z}_M, q\ne 0} \left|\sum_{\ell\in A_{k_t}} e^{2\pi i \ell q/M}\right| < C'\log M \right\rbrace \ge 1 -  \frac{1}{M^{\frac{C'}{2\sqrt{2}} - 2}}.
\end{equation*}
\noindent In particular, $$\mathbb{P}\left(\max_{\substack{j_t, j_{t+1}\in \mathbb{Z}_M, \\ j_t\ne j_{t+1}}}\left|\sum_{\ell_t\in A_{k_t}} e^{\frac{2\pi i}{M}\ell_t (j_t - j_{t+1})}\right| < C'\log M\right) \ge 1 -  \dfrac{1}{M^{\frac{C'}{2\sqrt{2}} - 2}}.$$ By taking the union bound over all $t\in \{1,\dots, m\}$, we conclude that
\begin{equation*}
\mathbb{P}\left(\left|\prod_{t = 1}^m \sum_{\ell_t\in A_{k_t}} e^{\frac{2\pi i}{M}\ell_t (j_t - j_{t+1})}\right| < C'^m \log^m M\right) \ge {1-\dfrac{m}{M^{\frac{C'}{2\sqrt{2}}-2}}}.
\end{equation*}
Then, applying the triangular inequality to the trace formula, we obtain that, on an event $X$ of probability at least $1 -  \dfrac{m}{M^{\frac{C'}{2\sqrt{2}} - 2}}$,
\begin{equation*}
\begin{split}
\mathbb{E}\left(\Tr H^m \right) & \le \sum_{\substack{j_1, j_2, \dots, j_m\in \mathbb{Z}_M,\\j_1\ne j_2\ne\dots\ne j_m\ne j_{1}}}\sum_{k_1, k_2, \dots, k_m\in \mathbb{Z}_M} E_{\substack{j_1\dots j_m\\k_1\dots k_m}} \left|\prod_{t = 1}^m \sum_{\ell_t\in A_{k_t}} e^{\frac{2\pi i}{M}\ell_t (j_t - j_{t+1})}\right| \\
& < C'^m \log^m M \sum_{\substack{j_1, j_2, \dots, j_m\in \mathbb{Z}_M,\\j_1\ne j_2\ne\dots\ne j_m\ne j_{1}}}\sum_{k_1, k_2, \dots, k_m\in \mathbb{Z}_M} E_{\substack{j_1\dots j_m\\k_1\dots k_m}}
\end{split}
\end{equation*}

A permutation $\alpha\in \Sigma_m$ can be presented as a product
\begin{equation}\label{eq_cycle_decomp}
\alpha = (i_{11} i_{12}\dots i_{1 r_1})(i_{21} i_{22}\dots i_{2 r_2})\dots (i_{s1} i_{s2}\dots i_{s r_s})
\end{equation}
\noindent of disjoint cycles, where $r_1 + r_2 +\dots + r_s = m$, and, for each $p\in \{1,\dots , s\}$, $\alpha(i_{p q}) = i_{p (q+1)}$ for $q\in \{1,\dots, r_p - 1\}$ and $\alpha(i_{p r_p }) = i_{p 1}$.

Suppose that we have $k_1,\dots , k_m$ fixed. Then $E_{\substack{j_1\dots j_m\\k_1\dots k_m}}\ne 0$ if and only if there exists $\alpha\in \Sigma_m$, such that $j_t -  j_{\alpha(t)} = k_t - k_{\alpha(t) - 1}$, for all $t\in \{1,\dots,m\}$. 
Assuming that $\alpha$ has $s$ cycles in the disjoint cycle decomposition \eqref{eq_cycle_decomp}, this condition can be rewritten in the form of $s$ systems of linear equations for $j_1,\dots, j_m$. Namely, for each $p\in \{1,\dots , s\}$, we have 
\begin{align*}
j_{i_{p1}} -  j_{i_{p2}} & = k_{i_{p1}} - k_{i_{p2} - 1}\\
j_{i_{p2}} -  j_{i_{p3}} & = k_{i_{p2}} - k_{i_{p3} - 1}\\
& \cdots\\
j_{i_{pr_p}} -  j_{i_{p1}} & = k_{i_{pr_p}} - k_{i_{p1} - 1}.\numberthis\label{system_for_j}
\end{align*}  
Note that the system \eqref{system_for_j} has rank $r_p-1$. Furthermore, summing up all the equations, on the left hand side we obtain zero. So, \eqref{system_for_j} has $M$ different solutions~if
\begin{equation}\label{eq_restrict_k}
\sum_{q = 1}^{r_p}k_{i_{pq}} = \sum_{q = 1}^{r_p}k_{i_{pq} - 1},
\end{equation}
and does not have a solution otherwise. Moreover, if $s\ne 1$, that is, $r_p<m$, then the sets of indices $\{i_{pq}\}_{q = 1}^{r_p}$ on the left hand side of \eqref{eq_restrict_k} and $\{i_{pq} - 1\}_{q = 1}^{r_p}$ on the right hand side of \eqref{eq_restrict_k} are different. Indeed, suppose that $\{i_{pq}\}_{q = 1}^{r_p} = \{i_{pq} - 1\}_{q = 1}^{r_p}$, and let $i_{pq_0} = \min_{q \in \{1,\dots, r_p\}} i_{pq}$ be the smallest element in this set. Since $i_{pq_0} - 1$ is also an element of $\{i_{pq}\}_{q = 1}^{r_p}$, we have $i_{pq_0} - 1\ge i_{pq_0}$, which implies $i_{pq_0} = 1$ and $i_{pq_0} - 1= m$. Then, since $m\in \{i_{pq}\}_{q = 1}^{r_p}$, we also have  $m -1\in \{i_{pq}\}_{q = 1}^{r_p}$. Proceeding the argument by induction, we obtain $\{i_{pq}\}_{q = 1}^{r_p} = \{1, \dots, m\}$, which is a contradiction. Without loss of generality, we can assume that $i_{pr_p}\notin \{i_{pq} - 1\}_{q = 1}^{r_p}$, for every $p\in \{1,\dots,s\}$. 

It follows that, for each cycle in the cycle decomposition~\eqref{eq_cycle_decomp}, except the last one, equation~\eqref{eq_restrict_k} is a nontrivial linear relation for $k_t$, $t \in \{1,\dots, m\}$. For the last cycle the relation follows automatically, assuming \eqref{eq_restrict_k} is satisfied for each $p\in \{1,\dots, s-1\}$. So, for the system of linear equations for $j_1,\dots, j_m$ to have a solution, $k_{i_{pr_p}}$, $p\in \{1,\dots , s-1\}$, should be determined by $\{k_1,\dots, k_m\}\setminus \{k_{i_{pr_p}}\}_{p = 1}^{ s-1}$ using equations \eqref{eq_restrict_k}. It this case the number of different solutions is~$M^s$.

Then, for the expectation of the trace of $H^m$, on the event $X$ we have
\begin{equation*}
\begin{split}
\mathbb{E}\left(\Tr H^m \right) & < C'^m \log^m M \sum_{\substack{j_1, j_2, \dots, j_m\in \mathbb{Z}_M,\\j_1\ne j_2\ne\dots\ne j_m\ne j_{1}}}\sum_{k_1, k_2, \dots, k_m\in \mathbb{Z}_M} E_{\substack{j_1\dots j_m\\k_1\dots k_m}}\\
& \le C'^m \log^m M \sum_{s = 1}^m S(m,s)  \sum_{j_{i_{11}},\dots, j_{i_{s1}}\in \mathbb{Z}_M} \sum_{\substack{k_{i_{11}},\dots, k_{i_{1(r_1-1)}}\in \mathbb{Z}_M \\ ^{\vdots} \\ k_{i_{(s-1)1}},\dots, k_{i_{(s-1)(r_s-1)}}\in \mathbb{Z}_M \\ k_{i_{s1}},\dots, k_{i_{sr_s}}\in \mathbb{Z}_M}} \frac{1}{M^m}\\
& = C'^m\frac{\log^m M}{M^m} \sum_{s = 1}^m S(m,s) M^s M^{m-s+1} \\
& = C'^m M\log^m M \sum_{s = 1}^m S(m,s) = C'^m m! M\log^m M,
\end{split}
\end{equation*}
\noindent where $S(m, s)$ denotes the Stirling number of the first kind that is equal to the number of permutations in $\Sigma_m$ with exactly $s$ cycles in the disjoint cycle decomposition.

Moreover, the cardinality of $\Lambda$ is given by a sum of $M^2$ independent Bernoulli random variables with success probability $\tau = \frac{C\log M}{M^{\frac{m-1}{m}}}$. More precisely, $$|\Lambda| = \sum_{(k,\ell)\in \mathbb{Z}_M\times \mathbb{Z}_M}{\bf 1}_{\Lambda}(k,\ell).$$ Then Hoeffding's inequality (Lemma \ref{Hoeffging_ineq_Bernoulli}) applied with $t = \frac{C\log M}{2M^{\frac{m-1}{m}}}$ implies
\begin{equation*}
\mathbb{P}\left\lbrace |\Lambda|\le \frac{1}{2}CM^{1+ \frac{1}{m}}\log M \right\rbrace\le e^{-2C^2M^{\frac{2}{m}}\log^2 M}.
\end{equation*}
\noindent That is, $|\Lambda| > \frac{1}{2}CM^{1+ \frac{1}{m}}\log M$ on an event $Y$ of probability at least $1 - e^{-2C^2M^{\frac{2}{m}}\log^2 M}$.

Then, on the event $X\cap Y$, which has probability at least $1 - \frac{\tilde{C}m}{M^{\frac{C'}{2\sqrt{2}} - 2}}$, for some $\tilde{C}>0$, the obtained estimates for the trace expectation and frame set cardinality lead to the following probability bound for the singular values estimates.
\begin{equation*}
\begin{split}
\mathbb{P}\left\lbrace\frac{|\Lambda|}{M}(1-\delta) \le A_{(g, \Lambda)} \le B_{(g, \Lambda)}\le \frac{|\Lambda|}{M}(1+\delta)\right\rbrace \geq 1 - \frac{M^{m}}{|\Lambda|^{m}}\delta^{-m}\mathbb{E}(\Tr H^{m})\\
\geq 1 - C'^m m! \delta^{-m} \frac{M^{m}}{\frac{1}{2^m}C^m M^{m+1}\log^m M}M\log^m M = 1 - \left(\frac{2C'}{C}\right)^m m! \delta^{-m}.
\end{split}
\end{equation*}
This concludes the proof, provided $C$ is chosen to be large enough.
\end{proof}

\begin{remark}
We note that the bounds obtained in Theorem~\ref{th_sing_val_rand_lambda} show the same asymptotic behavior as the bounds on the extreme singular values of matrices with independent entries obtained in~\cite{latala2005some,vershynin2018high}. This observation suggests that, for most of the choices of the frame set $\Lambda$, random time-frequency structured matrices are nearly as well-conditioned, as random matrices with independent Gaussian entries.
\end{remark}

\section{Numerical robustness to erasures for Gabor and MUB frames}\label{sec: deterministic NERF}
In this section we focus on the signal reconstruction in the case when a portion of the frame coefficients is lost during the measurement or transmission process. To ensure stable reconstruction, we require the original frame $\Phi$ to be numerically erasure-robust in the sense of Definition~\ref{def: NERF}. In other words, we need to establish uniform bounds $A_{\Phi'}$ and $B_{\Phi'}$ for all subframes $\Phi' \subset \Phi$ of a given size. 

We start our discussion by considering a Gabor frame $(g,\Lambda)$ with an arbitrary $g \in \mathbb{S}^{M-1}$. 
Clearly, in this case $(g,\Lambda)$ is a unit-norm frame and we can use Proposition~\ref{pr:lowerboundsmallestsing} to obtain a uniform lower bound on $A_{(g, \Lambda)}$ in the case when $\Lambda$ is a subgroup of $\Z_M \times \Z_M$. Indeed, when $\Lambda$ is a subgroup, Lemma~\ref{lm:timeshiftfrequencyshiftcommuteuptophase} implies that the sequence $\left( d_{\Lambda}[j]\right)_{j = 1}^{\vert \Lambda \vert}$ consisting of the elements of the set $\left\lbrace \left\vert\langle \pi(\lambda) g, \pi(\mu) g\rangle\right\vert^2\right\rbrace_{\lambda \in \Lambda}$ sorted in decreasing order is identical for each $\mu \in \Lambda$. For $\Lambda' \subset \Lambda$, we follow the trace estimation idea from \cite{fickus2012numerically} and use Proposition~\ref{pr:lowerboundsmallestsing} to immediately obtain
    \begin{equation*}
        A_{(g, \Lambda')} \geq \frac{\vert\Lambda \vert}{M} - \frac{M - 1}{2M} - \frac{1}{2}\left(\frac{\vert \Lambda\vert ^2}{M} - |\Lambda'|\sum_{j = 1}^{|\Lambda|'} d_{\Lambda}[j]\right). 
    \end{equation*}
In particular, for a full Gabor frame with $\Lambda = \mathbb{Z}_M\times \mathbb{Z}_M$, this bound gives the following result.

\begin{corollary}\label{cor: Gabor NERF}
    For a fixed $p\in (0,1]$, the full Gabor frame $(g, \mathbb{Z}_M\times \mathbb{Z}_M)$ is $(p,C)$-numerically erasure-robust with $$C = \frac{M^{1/2}}{\left( M - \frac{M-1}{2M} - \frac{M^2}{2}(1-p)\left( (1-p)M - \sum_{j = 1}^{(1-p)M^2}d_{\mathbb{Z}_M \times \mathbb{Z}_M}(j)\right)  \right)^{1/2}}.$$
\end{corollary}

Note that this bound is only meaningful if $$ M - \frac{M-1}{2M} - \frac{M^2}{2}(1-p)\left( (1-p)M - \sum_{j = 1}^{(1-p)M^2}d_{\mathbb{Z}_M \times \mathbb{Z}_M}(j)\right) > 0,$$ which depends on the values of $d_{\mathbb{Z}_M \times \mathbb{Z}_M}(j)$, and thus on the properties of the window $g$. A natural question therefore is if this bound can be further refined in the case when $g\in \mathbb{C}^M$ is a random, e.g. Steinhaus, window.

Unfortunately, the result of Corollary~\ref{cor: m2 Steinhaus general} is not strong enough to get a uniform robustness to erasures bound for the full Gabor frame even in the case of just one erasure. Indeed, in this case we have $\alpha = 1$ and $p = \frac{1}{M^2}$. The erased frame vector can be chosen in $M^2$ different ways, and thus taking the union bound over all the resulting subframes yields 
    \begin{align*}
        \mathbb{P} & \left(\frac{M^2 - 1}{M}(1-\delta) \le A_{(g, \Lambda)} \le B_{(g, \Lambda)}\le \frac{M^2 - 1}{M}(1+\delta) \text{ for all } \Lambda\subset \mathbb{Z}_M \times \mathbb{Z}_M, \vert \Lambda \vert = M^2 - 1 \right) \\ & \ge 1 - M^2 \frac{p}{\alpha(1 - p)}\frac{1}{\delta^2} = 1 - M^2 \frac{M^{-2}}{1 - M^{-2}} \frac{1}{\delta^2} = 1 - \frac{1}{1 - M^{-2}} \frac{1}{\delta^2}.
    \end{align*}
    Since $1 - \frac{1}{1 - M^{-2}} \frac{1}{\delta^2} < 0$, this bound is trivial. At the same time, we clearly have $$M - 1 \leq A_{(g, \Lambda)} \leq B_{(g, \Lambda)} \leq M \text{ for all } \vert \Lambda \vert = M^2 - 1.$$

\medskip

This observation suggests that further improvement of the robustness to erasures bound for Gabor frames obtained in Corollary~\ref{cor: Gabor NERF} requires development of new methods and approaches. We now turn our discussion to the analysis of the numerical robustness to erasures of a more general class of \emph{mutually unbiased bases frames} (MUBs). 

\begin{definition}[MUB frames]
    A frame $\Phi$ in $\mathbb{C}^M$ is said to be an $m$-MUB ($m$-mutually unbiased) frame if it is a union of $m$ orthonormal bases with coherence at most $\frac{1}{\sqrt{M}}$.
\end{definition}

This class is related to the Gabor frames in the following way. For prime ambient dimension $M$, there are known constructions of MUB frames as (deterministic) Gabor frames. 

\begin{theorem}[\cite{alltop1980complex}]
     Let $M \geq 5$ be prime and let $g_A$ with $g_A(j) = \frac{1}{\sqrt{M}}e^{2 \pi i j^3/ M}$ for $j \in \Z_M$ be an \emph{Alltop window}. Then the Gabor frame $(g_A, \Z_M \times \Z_M)$ is an $M$-MUB frame in $\C^M$. Moreover, the union $(g_A, \Z_M \times \Z_M)\cup \{e_j\}_{j = 1}^M$, where $\{e_j\}_{j = 1}^M$ is the standard orthonormal basis, is an $(M + 1)$-MUB, that is, a mutually unbiased frame of maximal cardinality. 
\end{theorem}

\subsection{Robustness to erasures of MUB Frames}
In this section, we significantly improve the numerical robustness to erasures result for MUB frames obtained in \cite{fickus2012numerically} and show that MUB frames are robust to up to 50$\%$ erasures. This is comparable with the strongest guarantees of this kind obtained in~\cite{fickus2012numerically} for equiangular tight frames, see Theorem~\ref{th:ETF}.
The result for MUB frames from \cite{fickus2012numerically} is given by the following theorem. Roughly, it states that when the size of the MUB frame is $M^2$, one can afford to lose $O(M)$ frame coefficients. 

\begin{theorem}[{\cite[Theorem~6]{fickus2012numerically}}]
    Let $\Phi$ be an $M$-MUB frame. Then $\Phi$ is a $(p,C)$-numerical erasure-robust frame for any $p \leq \frac{(C^2 - 1)^2}{(C^2 + 1)(M + 1)}$.  
\end{theorem}

We refine the proof technique of~\cite[Theorem~5]{fickus2012numerically} and obtain the following stronger result, which is similar to Theorem~\ref{th:ETF} for equiangular tight frames. 
\begin{theorem}\label{th:MUBGuarantee}
    Let $\Phi$ be an $m$-MUB frame with $m = \alpha M$. Then $\Phi$ is a $(p,C)$-numerically erasure robust frame for any $p \leq \frac{\alpha (C^2 - 1)^2}{\alpha(C^2 - 1)^2 + (C^2 + 1)^2}$.
    \begin{proof}
    Let $\mathcal{J} \subset \{1,\dots,mM\}$ be of size $J = (1 - p)mM = \alpha(1-p)M^2$ and let $\Phi_{\mathcal{J}}$ be the associated subframe. As $\Phi$ is an $m$-MUB, we may write $\Phi = \{\varphi_j^{(i)} : 1 \leq j \leq M, 1 \leq i \leq m\}$, where for each $i$, $\{\varphi_j^{(i)}\}_{j = 1}^M$ is an orthonormal basis. Let us define $A_i = \{j | \varphi_j^{(i)} \in \Phi_{\mathcal{J}}\}$. Clearly, we have that $\sum_{i = 1}^m |A_i| = J$. By definition of an $m$-MUB frame, the inner products $$|\langle \varphi_j^{(i)} , \varphi_{\tilde{j}}^{(\tilde{i})}\rangle|^2 = \begin{cases}
        0, & i = \tilde{i}, j\neq\tilde{j};\\
        1, & i = \tilde{i}, j = \tilde{j};\\
        \frac{1}{M}, & i \neq \tilde{i}.
    \end{cases}$$ 
    The idea is to count how many of each of these instances occure in the frame potential $\mathrm{FP}(\Phi_{\mathcal{J}})$. It is clear that the value $1$ is taken on precisely $J$ times. We observe that the value $0$ occurs when we have the inner product of two different vectors belonging to the same $A_{i}$. Thus, the total number of zero summands in the formula for the frame potential is given by 
    \begin{align*}
        \sum_{i = 1}^m |A_i|\left(|A_i| - 1\right) = \sum_{i = 1}^m |A_i|^2 - J.
    \end{align*} The occurrences of $\frac{1}{M}$ are therefore given by 
    \begin{align*}
    J^2 - \left( \sum_{i = 1}^m |A_i|^2 - J\right) - J = J^2 -\sum_{i = 1}^m |A_i|^2.
    \end{align*} Thus,

    \begin{align}\label{eq:tracealltop}
        \mathrm{FP}(\Phi_{\mathcal{J}}) - \frac{J^2}{M} &= \frac{J^2}{M} - \frac{1}{M} \sum_{i = 1}^m |A_i|^2 + J - \frac{J^2}{M}
        = J - \frac{1}{M} \sum_{i = 1}^m |A_i|^2. 
    \end{align}
    Using the Cauchy-Schwarz inequality, we see that $\frac{1}{mM} J^2 \leq \frac{1}{M} \sum_{i = 1}^m |A_i|^2$. It follows that 

    \begin{align*}
        \mathrm{FP}(\Phi_{\mathcal{J}}) - \frac{J^2}{M} \leq  J - \frac{1}{mM}J^2 
    \end{align*}
    Consequently,

    \begin{align*}
        \delta_{\mathcal{J}}^2 \leq \frac{M^2}{J^2} \left(J - \frac{1}{mM}J^2\right) = \frac{M^2}{J}\left(1 - \frac{1}{mM}J\right),
    \end{align*}
    where $\delta_{\mathcal{J}} := \delta_{\Phi_{\mathcal{J}}}$ as in the proof of Lemma \ref{lem: trace}. Now we use that $J = (1-p)mM = \alpha (1 - p)M^2$ to obtain
\begin{align}\label{eq:deltaest}
        \delta_{\mathcal{J}}^2 \leq \frac{p}{\alpha(1 - p)}.
\end{align}

    By the theorem assumption, we have that $p \leq \frac{\alpha(C^2 - 1)^2}{\alpha (C^2 - 1)^2 + (C^2 + 1)^2}$. Substituting this into~\eqref{eq:deltaest} yields
    $\delta_{\mathcal{J}}^2 \leq \frac{(C^2 - 1)^2}{(C^2+ 1)^2}$, implying $\mathrm{Cond} (\Phi^*_{\mathcal{J}}) \leq C$.  
    \end{proof}
\end{theorem}

Unlike the constructions of maximal equiangular tight frames, constructions for maximal MUB frames are known in many ambient dimensions (including powers of prime numbers). Note that for a maximal MUB frame we can take $\alpha = 1 + \frac{1}{M}$. If we let $C \to \infty$, we see that we can provide a robustness guarantee of up $\frac{M + 1}{2M + 1} > \frac{1}{2}$ erasures, thereby making a small step towards breaking the ``one-half barrier'' described in~\cite{fickus2012numerically}. Maximal ETFs would also achieve the same guarantee. 

\begin{remark}
    Observe that the expected value of the trace of $H_\Lambda^2$ obtained in Proposition~\ref{pr:tracesteinhaus} is precisely the expression for the trace when $g$ is an Alltop window obtained by \eqref{eq:tracealltop} and Corollary~\ref{cr:fpnerf}. This leads to the belief that Gabor frames with a Steinhaus window are also numerically robust against erasures. 
\end{remark}

\section{Numerical results and further discussion}\label{sec_num_res_sing_val}

In this section, we aim to numerically analyze the obtained theoretical guarantees for frame bounds and the NERF property for Gabor frames. In particular, we discuss Theorem \ref{th_sing_val_rand_lambda} and Theorem \ref{th:MUBGuarantee}. 

\smallskip

Recall that Theorem \ref{th_sing_val_rand_lambda} gives estimates of the frame bounds of a Gabor frame $(g,\Lambda)$ with a Steinhaus window $g$ and $\Lambda$ being a random subset of $\Z_M \times \Z_M$. 
Let us fix an even $m\in\mathbb{N}$, and let $C>0$ be a sufficiently large constant depending on $m$. Consider a random subset $\Lambda\subset \mathbb{Z}_M\times\mathbb{Z}_M$, such that the events $\{(k,\ell)\in \Lambda\}$ are independent for all $(k,\ell)\in \mathbb{Z}_M\times \mathbb{Z}_M$ and have probability ${\tau = \frac{C\log M}{M^{\frac{m-1}{m}}}}$. Theorem~\ref{th_sing_val_rand_lambda} ensures that
\begin{equation*}
\mathbb{P}\left\lbrace\frac{|\Lambda|}{M}(1-\delta) \le A_{(g, \Lambda)} \le B_{(g, \Lambda)}\le \frac{|\Lambda|}{M}(1+\delta)\right\rbrace\ge 1 - \varepsilon,
\end{equation*}
\noindent where $\varepsilon\in (0,1)$ depends on $m$, $\delta$, and the choice of $C$. To illustrate Theorem \ref{th_sing_val_rand_lambda}, we use two sets of numerical simulations. 

\begin{figure}[t]\center
\begin{tabular}{cc}
\includegraphics[width=85mm]{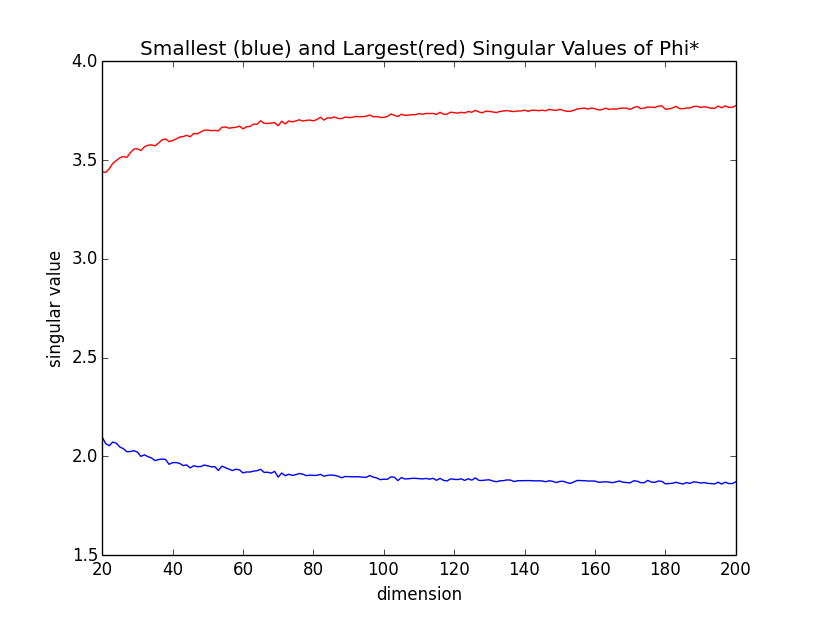}
&
\includegraphics[width=85mm]{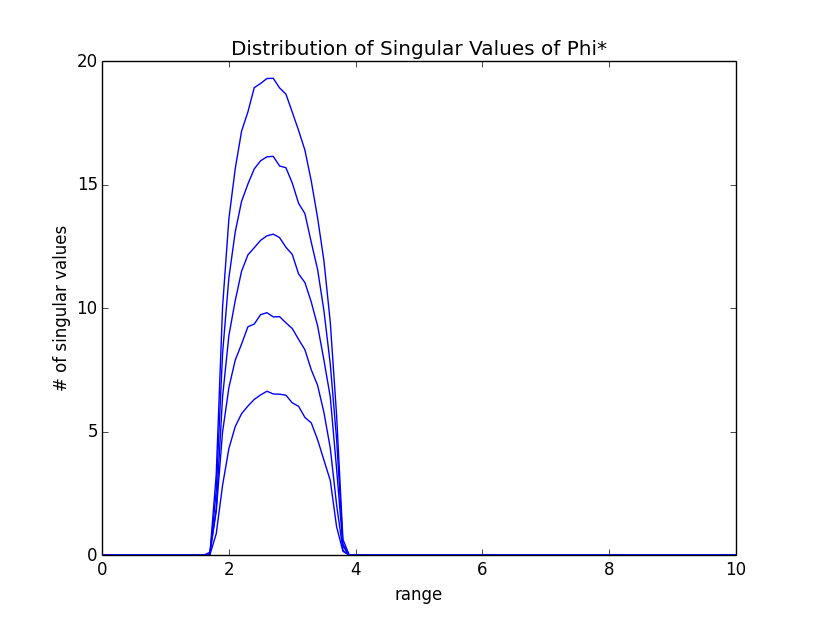}
\end{tabular}
\caption{\label{num_sing_val_distr} Left: the dependence of the upper and lower frame bounds of a Gabor frame $(g, \Lambda)$ on the ambient dimension $M$; Right: the distribution of the singular values of the analysis matrix of $(g, \Lambda)$ for  $M = 100, 150, 200, 250, 300$. On both figures, $g$ is a Steinhaus window and $\Lambda$ is chosen at random as described in Theorem \ref{th_sing_val_rand_lambda}, with $\tau = \frac{C}{M}$, that is, $|\Lambda| = O(M)$ with high probability. The plots are obtained by averaging over 1000 randomly generated frames.}
\end{figure}

In the first set of numerical simulations, we investigate the behavior of the singular values of the analysis matrix of a Gabor frame $(g, \Lambda)$ with a Steinhaus window $g$ and set $\Lambda\subset \mathbb{Z}_M\times\mathbb{Z}_M$ selected at random, so that $|\Lambda| = O(M)$ with high probability. The obtained numerical results suggest that, in the case when random $\Lambda$ is constructed as described in Theorem \ref{th_sing_val_rand_lambda} with $\tau = \frac{C}{M}$, there exist constants $0<k<K$ not depending on the ambient dimension $M$, such that all the singular values of the analysis matrix $\Phi_{\Lambda}^*$ are inside the interval $\left[k\frac{|\Lambda|}{M}, K\frac{|\Lambda|}{M}\right]$ with high probability, see Figure \ref{num_sing_val_distr} (left). This allows us to conjecture that a version of Theorem~\ref{th_sing_val_rand_lambda} is true also for~$\Lambda$ with $|\Lambda| = O(M)$, and the additional factor of $M^\epsilon \log M$ in the cardinality of $\Lambda$ is a side effect of the method used to prove the theorem. 
The right-hand side of Figure \ref{num_sing_val_distr} shows the distribution of the singular values of $\Phi_{\Lambda}^*$ over this interval for the selected dimensions $M = 100,~150,~200,~250,~300$. 

\medskip

\begin{figure}[t]\center
\begin{tabular}{cc}
\includegraphics[width=83mm]{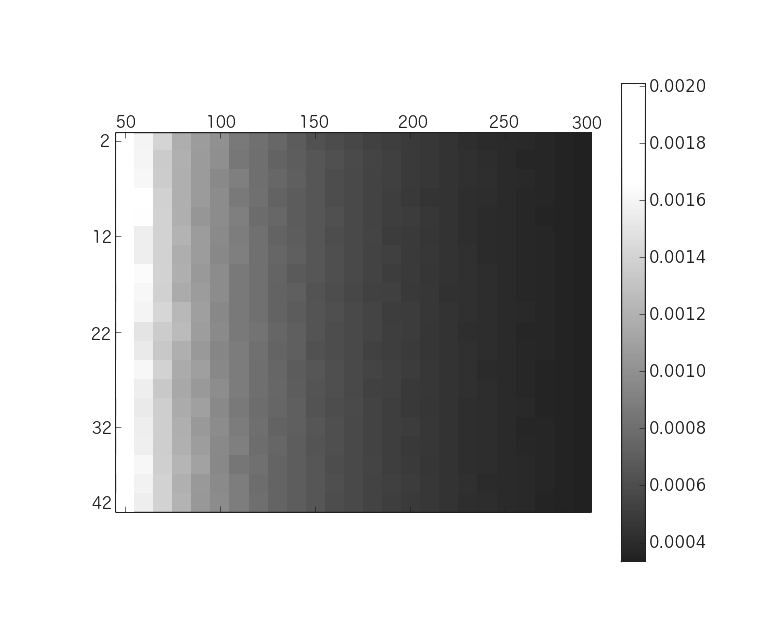}
&
\includegraphics[width=83mm]{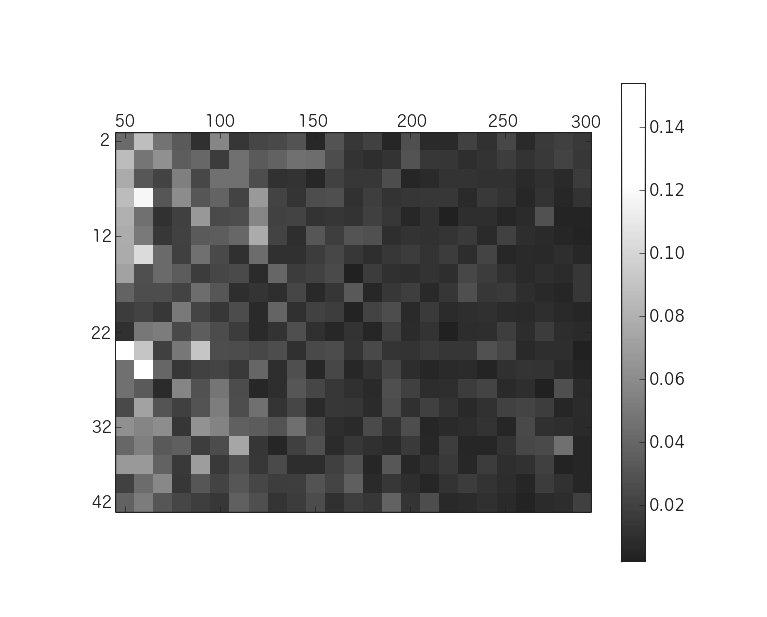}
\end{tabular}
\caption{\label{num_sing_val_trace_est} The dependence of the numerically estimated normalized trace expectation $\frac{M^{m}}{|\Lambda|^{m}}\mathbb{E}\left(\Tr \left(\Phi_\Lambda \Phi_\Lambda ^*- \frac{|\Lambda|}{M}I_M \right)^{m}\right)$ on the ambient dimension $M$ (horizontal axis) and the parameter $C$ (vertical axis), for~$m = 4$. Here, $\Phi_\Lambda$ is the synthesis matrix of a Gabor frame $(g, \Lambda)$ with a Steinhaus window $g$. Left: $\Lambda$ is chosen at random, as described in Theorem~\ref{th_sing_val_rand_lambda}, with $\tau = \frac{C}{M}$. Right: ${\Lambda = F\times \{0, 1, \dots, \lfloor\frac{M}{2}\rfloor\}}$ with $\vert F \vert = 2C$. 
}
\end{figure}

We use the second set of simulations to investigate the behavior of the trace of the matrix $H = \Phi_\Lambda \Phi_\Lambda ^*- \frac{|\Lambda|}{M}I_M$, where $\Phi_\Lambda$ is the synthesis matrix of a Gabor frame $(g, \Lambda)$ with a Steinhaus window $g$. It follows from Lemma~\ref{lemma_trace_formula} that $${\mathbb{P}\left\lbrace A_{(g, \Lambda)} \le \frac{|\Lambda|}{M}(1-\delta)  \text{ or } B_{(g, \Lambda)}\ge \frac{|\Lambda|}{M}(1+\delta)\right\rbrace \leq \frac{M^{2m}}{|\Lambda|^{2m}}\delta^{-2m}\mathbb{E}(\Tr H^{2m})}.$$ In other words, the normalized trace expectation $\frac{M^{m}}{|\Lambda|^{m}}\mathbb{E}\left(\Tr H^{m}\right)$ is used to estimate the probability of the ``failure'' event on which either the lower frame bound of $(g, \Lambda)$ is too small or its upper frame bound is too large, meaning that the frame $(g, \Lambda)$ is not well-conditioned.

For the normalized trace expectation, we consider two different constructions of~$\Lambda$, providing the average and the ``worst-case'' estimates, respectively. The left-hand side of Figure~\ref{num_sing_val_trace_est} shows the numerical results in the case when $\Lambda$ is chosen at random, as described in Theorem~\ref{th_sing_val_rand_lambda} with $\tau = \frac{C}{M}$. The right-hand side of Figure~\ref{num_sing_val_trace_est} illustrates the case when $\Lambda$ is of the form $\Lambda = F\times \{0, 1, \dots, \lfloor\frac{M}{2}\rfloor\}$, $F\subset\mathbb{Z}_M$. Indeed, following~\eqref{eq: trace expression}, we see that
\begin{equation*}
\mathbb{E}\left(\Tr H^m \right) = \sum_{\substack{j_1, j_2, \dots, j_m\in \mathbb{Z}_M,\\j_1\ne j_2\ne\dots\ne j_m\ne j_{1}}}\sum_{k_1, k_2, \dots, k_m\in \mathbb{Z}_M} E_{\substack{j_1\dots j_m\\k_1\dots k_m}} \prod_{t = 1}^m \sum_{\ell_t\in A_{k_t}} e^{\frac{2\pi i}{M}\ell_t (j_t - j_{t+1})},
\end{equation*}
where $E_{\substack{j_1\dots j_m\\k_1\dots k_m}}\in\left\lbrace 0,\frac{1}{M^m}\right\rbrace$ and $A_{k} = \{\ell\in \mathbb{Z}_M ~\colon~ (k,\ell)\in \Lambda\}$.
To maximize the expected trace, one needs to select $\Lambda$ of the given cardinality $CM$ in a way that maximizes the values of $\sum_{\ell\in A_{k}} e^{\frac{2\pi i}{M}\ell j}$. The choice $\Lambda = F\times \{0, 1, \dots, \lfloor\frac{M}{2}\rfloor\}$ implies that $A_k = \{0, 1, \dots, \lfloor\frac{M}{2}\rfloor\}$ for all $k$ and thus ensures that the summands in the sum are localized.

For each of the constructions of $\Lambda$, Figure~\ref{num_sing_val_trace_est} shows the dependence of the normalized trace expectation on the ambient dimension $M$ (horizontal axis) and the parameter $C$ (vertical axis), for~$m = 4$. The obtained numerical results suggest that, in both cases, the normalized trace expectation decreases rapidly with the dimension. This allows us to conjecture that the probability bound obtained in Theorem~\ref{th_sing_val_rand_lambda} can be further improved and extended to smaller $\vert \Lambda \vert$. Moreover, Figure \ref{num_sing_val_trace_est} (left) shows that, in the case of randomly selected $\Lambda$, the normalized trace expectation does not seem to depend on the parameter $C$.

\subsection{Erasure-robust frames}

We now turn our attention to the numerical investigation of the robustness to erasures of Gabor frames~$(g, \Lambda)$ with a random window $g$, as well as  mutually unbiased bases frames.

We note that, for any $\Lambda'\subset\Lambda$, 
\begin{equation*}
B_{(g, \Lambda')} = \max_{x\in \mathbb{S}^{M-1}}\sum_{\lambda\in \Lambda'}|\langle x, \pi(\lambda)g\rangle|^2\le \max_{x\in \mathbb{S}^{M-1}}\sum_{\lambda\in \Lambda}|\langle x, \pi(\lambda)g\rangle|^2 =B_{(g, \Lambda)},
\end{equation*}
\noindent and, in particular, for any $\Lambda\subset \Z_M\times\mathbb{Z}_M$ and $g\in \mathbb{S}^{M-1}$, $B_{(g, \Lambda)} \leq M$.
Thus, we concentrate on uniformly bounding the lower frame bound  $A_{(g, \Lambda')}$, for all subframes $(g, \Lambda')$ of $(g, \Lambda)$ with $|\Lambda'| \ge (1 - p)|\Lambda|$, where $p$ is a fixed portion of erasures.

\begin{figure}[t]\center
\includegraphics[width=115mm]{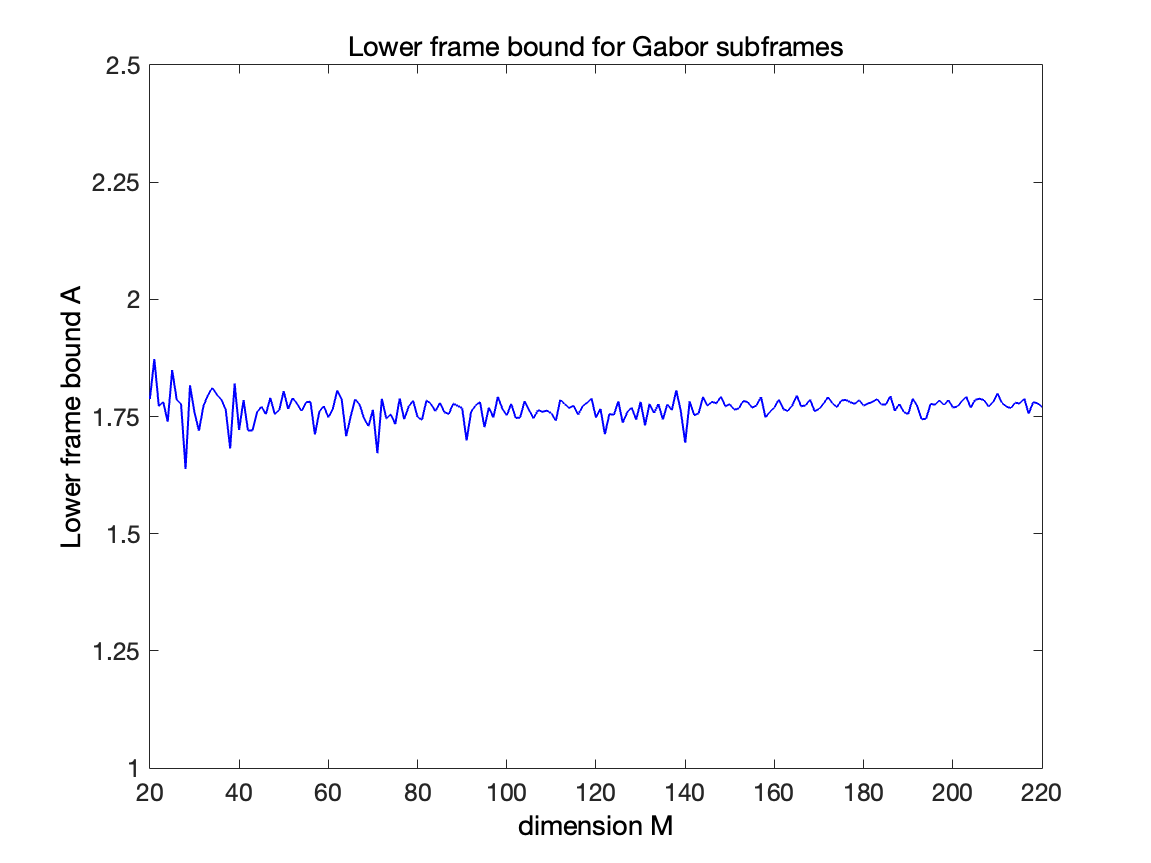}
\caption{\label{num_singular} The dependence of the numerically estimated ${\Delta\left(\frac{1}{3}\right)=\min\left\lbrace A_{(g,\Lambda')}:~ \Lambda'\subset \Lambda,~ |\Lambda'|\ge \frac{2}{3}|\Lambda|\right\rbrace}$ on the ambient dimension $M$. Here, $g\sim\text{Unif.}(\mathbb{S}^{M-1})$, and ${\Lambda = F\times\mathbb{Z}_M}$ with $\vert F\vert = \text{const}$ independent of $M$. For each dimension, the plot shows the smallest value of $A_{(g,\Lambda')}$ obtained over 1000 randomly selected $\Lambda'\subset \Lambda$ with $|\Lambda'|= \frac{2}{3}|\Lambda|$. 
}
\end{figure}

To this end, for each $p\in [0,1]$, let us define the following quantity 
\begin{equation*}
\Delta(p) = \min_{\substack{\Lambda'\subset \Lambda,\\ \vert\Lambda'\vert\ge (1 - p)\vert\Lambda\vert}}A_{(g,\Lambda')}.
\end{equation*}

\begin{remark}
    Note that following Definition~\ref{def: NERF}, a Gabor frame $(g, \Lambda)$ is an $(p, C)$-numerically erasure-robust frame with $C = \sqrt{\frac{M}{\Delta(p)}}$. Furthermore, if $\Lambda = F\times \mathbb{Z}_M$ and using the bounds for $B_{g, \Lambda}$ obtained in Example~\ref{example_structured}, we obtain that for a random Steinhaus window~$g$, $C = \sqrt{\frac{\vert F \vert}{\Delta(p)}}$ and for a window~$g\sim\text{Unif.}(\mathbb{S}^{M-1})$, $C = \sqrt{\frac{20\vert F \vert}{\Delta(p)}}$.
\end{remark}

Numerical results illustrating the dependence of the value $\Delta(p)$ (for $p = \frac{1}{3}$) on the dimension $M$ are presented in Figure~\ref{num_singular}. For each dimension, the plot shows the smallest value of $A_{(g,\Lambda')}$ obtained over 1000 randomly selected $\Lambda'\subset \Lambda$ with $|\Lambda'|= \frac{2}{3}|\Lambda|$. These numerical results suggest that $\Delta(p)$ is bounded away from zero by a numerical constant that is independent of~$M$, that is, the Gabor frame $(g, F\times \mathbb{Z}_M)$ is robust to erasures. 

\medskip

Next, we investigate the robustness-to-erasures guarantees we obtain for MUB-frames. Theorem~\ref{th:MUBGuarantee}, the proof of which relies on Lemma~\ref{lem: trace} with $m = 1$, states that maximal MUB-frames can achieve numerical robustness to erasures of up to 50\%. A natural question is if this result can be improved and robustness-to-erasures guarantees can be obtained for even higher erasure rates, e.g., by using higher values of $m$.

The bound we have obtained for MUB-frames in Theorem~\ref{th:MUBGuarantee} is independent of the ambient dimension. With this in mind, for the numerical experiments we set $M = 5$ and consider the full Gabor MUB-frame $(g_A, \Z_5 \times \Z_5)$, as it is possible to enumerate all its subframes. For each fixed erasure rate $p$, we consider all the subframes $(g_A, \Lambda')\subset (g_A, \Z_5 \times \Z_5)$ with $\vert \Lambda'\vert = (1 - p)M^2$ and compute their largest (worst-case) condition number $\max_{\substack{\Lambda'\subset \Z_5\times \Z_5\\\vert \Lambda '\vert = 25(1 - p)}} \text{Cond}(\Phi_{\Lambda'}^*)$, as well as its theoretical estimate provided by Theorem~\ref{th:MUBGuarantee} and largest over all subframes trace estimates for $m = 1$ and $m = 2$. The results are presented in Table~\ref{tbl: MUB numerics}

\begin{table}[hbt!]
\centering
\begin{tabularx}{0.97 \textwidth}{| X | X | X | X | X |}
\hline
$p$ & Estimate trace $m = 1$ & Estimate trace $m = 2$ & Theoretical\linebreak estimate & Worst-case \linebreak $\text{Cond} (\Phi_{\Lambda'}^*)$ \\
\hline
0.0 & 1.0 & 1.0 & 1.0 & 1.0 \\
\hline
0.04 & 1.207488 & 1.184004  & 1.230022 & 1.118034 \\
\hline
0.08 & 1.326102 & 1.265527  & 1.355143 & 1.186316 \\
\hline
0.12 & 1.444592 & 1.363902  & 1.473415 & 1.268861 \\
\hline
0.16 & 1.576014 & 1.458856  & 1.596509 & 1.35509 \\
\hline
0.2 & 1.732051 & 1.578976  & 1.732051 & 1.451066 \\
\hline
0.24 & 1.861272 & 1.673945  & 1.888307 & 1.535922 \\
\hline
0.28 & 2.027217 & 1.787106  & 2.076928 & 1.615618 \\
\hline
0.32 & 2.252406 & 1.938546  & 2.317178 & 1.728263 \\
\hline
0.36 & 2.584151 & 2.122931  & 2.645751 & 1.877075 \\
\hline
0.4 & 3.146264 & 2.35985  & 3.146264 & 2.049199 \\
\hline
0.44 & 3.869975 & 2.631544  & 4.075101 & 2.277497 \\
\hline
0.48 & 5.792162 & 3.041437  & 7.069653 & 2.579654 \\
\hline
0.52 & $\infty$ & 3.823597  & $\infty$ & 2.884371 \\
\hline
0.56 & $\infty$ & 6.345638  & $\infty$ & 3.517504 \\
\hline
0.6 & $\infty$ & $\infty$  & $\infty$ & 3.891432 \\
\hline
0.64 & $\infty$ & $\infty$  & $\infty$ & 4.703299 \\
\hline
0.68 & $\infty$ & $\infty$  & $\infty$ & 6.167132 \\
\hline
\end{tabularx}
\caption{For different values of the erasure rate $p$, the table shows the worst-case subframe condition number $\max_{\substack{\Lambda'\subset \Z_5\times \Z_5\\\vert \Lambda '\vert = 25(1 - p)}} \text{Cond}(\Phi_{\Lambda'}^*)$, as well as its theoretical estimate provided by Theorem~\ref{th:MUBGuarantee} and largest over all subframes trace estimates for $m = 1$ and $m = 2$. The value $\infty$ here means that there is no trace estimate due to $\delta_{\Phi} \geq 1$ or no theoretical guarantee with erasures of more than 50\%.}\label{tbl: MUB numerics}
\end{table}

\newpage

We observe that the trace estimate with $m = 2$ gives a bound on the worst-case subframe condition number for erasure rates higher than 50\%. Thus suggests that considering even higher values of $m$ can potentially allow one to obtain robustness-to-erasures guarantees for MUB-frames with even higher values of $p$. Remarkably, the true values of the worst-case subframe condition number seem to not exceed 7, even with an erasure rate of nearly 70\%.

\section*{Acknowledgments}

PS thanks Prof. Dr. Holger Rauhut for insightful discussions on different occasions. PS is supported by NWO Talent program Veni ENW grant, file number VI.Veni.212.176.

\printbibliography

@ARTICLE{balan1,
    AUTHOR = "Radu Balana and Peter Casazza and Dan Edidin",
    TITLE = "On signal reconstruction without phase",
    JOURNAL = {Applied and Computational Harmonic Analysis},
    VOLUME =  {20},
    NUMBER = {3},
    PAGES =  {pp.345-356},
    YEAR = 2006}

@book{vershynin2018high,
  title={High-dimensional probability: An introduction with applications in data science},
  author={Vershynin, Roman},
  volume={47},
  year={2018},
  publisher={Cambridge university press}
}

@INCOLLECTION{pfander2,
    AUTHOR = "G\text{\"{o}}tz E. Pfander",
    TITLE = "Gabor frames in finite dimensions",
    BOOKTITLE = "Finite Frames: Theory and Applications",
    EDITOR = "Peter G. Casazza and Gitta Kutyniok",
    PUBLISHER = {Birkhäser Boston},
    YEAR = 2013}

@book{tao,
  title={Additive combinatorics},
  author={Tao, Terence and Vu, Van H.},
  volume={13},
  year={2006},
  publisher={Cambridge University Press}
}

@ARTICLE{rudelson1,
  title={The {L}ittlewood--{O}fford problem and invertibility of random matrices},
  author={Rudelson, Mark and Vershynin, Roman},
  journal={Advances in Mathematics},
  volume={218},
  number={2},
  pages={600--633},
  year={2008},
  publisher={Elsevier}}

@ARTICLE{rudelson2,
  title={Smallest singular value of a random rectangular matrix},
  author={Rudelson, Mark and Vershynin, Roman},
  journal={Communications on Pure and Applied Mathematics},
  volume={62},
  number={12},
  pages={1707--1739},
  year={2009},
  publisher={Wiley Online Library}}

@article{laurent2000adaptive,
  title={Adaptive estimation of a quadratic functional by model selection},
  author={Laurent, B{\'e}atrice and Massart, Pascal},
  journal={Annals of Statistics},
  pages={1302--1338},
  year={2000},
  publisher={JSTOR}
}

@article{marsaglia1972choosing,
  title={Choosing a point from the surface of a sphere},
  author={Marsaglia, George and others},
  journal={The Annals of Mathematical Statistics},
  volume={43},
  number={2},
  pages={645--646},
  year={1972},
  publisher={Institute of Mathematical Statistics}
}

@article{krahmer2014suprema,
  title={Suprema of chaos processes and the restricted isometry property},
  author={Krahmer, Felix and Mendelson, Shahar and Rauhut, Holger},
  journal={Communications on Pure and Applied Mathematics},
  volume={67},
  number={11},
  pages={1877--1904},
  year={2014},
  publisher={Wiley Online Library}
}

@article{pfander2010sparsity,
  title={Sparsity in time-frequency representations},
  author={Pfander, G{\"o}tz E and Rauhut, Holger},
  journal={Journal of Fourier Analysis and Applications},
  volume={16},
  number={2},
  pages={233--260},
  year={2010},
  publisher={Springer}
}

@article{strohmer2003grassmannian,
  title={Grassmannian frames with applications to coding and communication},
  author={Strohmer, Thomas and Heath, Robert W.},
  journal={Applied and computational harmonic analysis},
  volume={14},
  number={3},
  pages={257--275},
  year={2003},
  publisher={Elsevier}
}

@article{fickus2012numerically,
  title={Numerically erasure-robust frames},
  author={Fickus, Matthew and Mixon, Dustin G.},
  journal={Linear Algebra and its Applications},
  volume={437},
  number={6},
  pages={1394--1407},
  year={2012},
  publisher={Elsevier}
}

@article{latala2005some,
  title={Some estimates of norms of random matrices},
  author={Lata{\l}a, Rafa{\l}},
  journal={Proceedings of the American Mathematical Society},
  volume={133},
  number={5},
  pages={1273--1282},
  year={2005}
}

@book{finite_frames_book,
  title={Finite Frames: Theory and Applications},
  author={Casazza, Peter G. and Kutyniok, Gitta},
  year={2013},
  publisher={Springer}
}

@article{alltop1980complex,
  title={Complex sequences with low periodic correlations (corresp.)},
  author={Alltop, W},
  journal={IEEE Transactions on Information Theory},
  volume={26},
  number={3},
  pages={350--354},
  year={1980},
  publisher={IEEE}
}

@article{benedetto2003finite,
  title={Finite normalized tight frames},
  author={Benedetto, John J and Fickus, Matthew},
  journal={Advances in Computational Mathematics},
  volume={18},
  pages={357--385},
  year={2003},
  publisher={Springer}
}

@incollection{kingsbury1998wavelet,
  title={Wavelet transforms in image processing},
  author={Kingsbury, Nick and Magarey, Julian},
  booktitle={Signal analysis and prediction},
  pages={27--46},
  year={1998},
  publisher={Springer}
}

@article{jaganathan2016phase,
  title={Phase retrieval: An overview of recent developments},
  author={Jaganathan, Kishore and Eldar, Yonina C and Hassibi, Babak},
  journal={Optical compressive imaging},
  pages={279--312},
  year={2016},
  publisher={CRC Press}
}

@article{innocenti2023shadow,
  title={Shadow tomography on general measurement frames},
  author={Innocenti, Luca and Lorenzo, Salvatore and Palmisano, Ivan and Albarelli, Francesco and Ferraro, Alessandro and Paternostro, Mauro and Palma, G Massimo},
  journal={PRX Quantum},
  volume={4},
  number={4},
  pages={040328},
  year={2023},
  publisher={APS}
}

@article{bammer2019gabor,
  title={Gabor frames and deep scattering networks in audio processing},
  author={Bammer, Roswitha and D{\"o}rfler, Monika and Harar, Pavol},
  journal={Axioms},
  volume={8},
  number={4},
  pages={106},
  year={2019},
  publisher={MDPI}
}

@article{rolland2010gabor,
  title={Gabor-based fusion technique for optical coherence microscopy},
  author={Rolland, Jannick P and Meemon, Panomsak and Murali, Supraja and Thompson, Kevin P and Lee, Kye-sung},
  journal={Optics express},
  volume={18},
  number={4},
  pages={3632--3642},
  year={2010},
  publisher={Optical Society of America}
}

\newpage

\appendix

\section{Appendix: Probability theory tools}\label{probability_background}

In this appendix, we collect the probabilistic tools and results used in the proofs of this paper. We start by stating the Hoeffding's inequality in the special case of Bernoulli random variables.

\begin{lemma}[Hoeffding's inequality]\label{Hoeffging_ineq_Bernoulli}
Let $X_j$, $j\in \{1,\dots N\}$, be independent identically distributed Bernoulli random variables, such that $\mathbb{P}\{X_j=1\} = p$, for some $p\in (0,1)$, that is $X_j \sim \text{i.i.d. } B\left(1,p\right)$. Consider the random variable ${S = \sum_{j = 1}^N X_j}$. Then, for every $t>0$, we have
\begin{equation*}
\begin{gathered}
\mathbb{P}\{S< (p - t)N\}\le e^{-2t^2N} \quad \text{and} \quad \mathbb{P}\{S> (p + t)N\}\le e^{-2t^2N}.
\end{gathered}
\end{equation*}
\end{lemma}

The following lemma, proven in \cite{laurent2000adaptive}, is useful for obtaining bounds on the norms of random vectors.

\begin{lemma}\emph{\textbf{\cite{laurent2000adaptive}}}\label{chi_square}
Let $Y_1,\dots,Y_M \sim  i.i.d. ~\mathcal{N}(0,1)$ and fix $c = (c_1,\dots,c_M)$ with $c_k\ge 0$, $k\in\{1,\dots, M\}$. Then, for $Z = \sum_{k = 1}^M c_k(Y_k^2 - 1)$ the following inequalities hold for any $t>0$.
\begin{equation}\label{lower}
\mathbb{P}\{Z \ge 2\Vert c\Vert _2 \sqrt{t} + 2\Vert c\Vert _\infty t\} \le e^{-t};
\end{equation}
\begin{equation}\label{upper}
\mathbb{P}\{Z \le -2\Vert c\Vert _2 \sqrt{t}\} \le e^{-t}.
\end{equation}

\end{lemma}

Using Lemma \ref{chi_square}, we obtain the following bounds on the norm of a random Gaussian vector $h\sim \mathcal{C}\mathcal{N}\left( 0, \frac{1}{M} I_M \right)$.

\begin{lemma}\label{lemma_norm_gaussian}
Consider a random vector $h\in \mathbb{C}^M$, such that $h\sim \mathcal{C}\mathcal{N}\left( 0, \frac{1}{M} I_M \right)$. Then, there exists a constant $C>0$, such that
\begin{equation*}
\mathbb{P}\left\lbrace\frac{1}{2} < \Vert h\Vert _2 < 2\right\rbrace \ge 1 - e^{-CM}.
\end{equation*}
\end{lemma}

\begin{proof}
First, we note that $$2M \Vert h\Vert _2^2 = 2M \sum_{k = 1}^M (|a_k|^2 + |b_k|^2),$$ where $h(k) = a(k) + ib(k)$ and $a(k), b(k)\sim i.i.d. ~ \mathcal{N}(0,\frac{1}{2M})$. Then, for any \linebreak$k \in \{1,\dots,M\}$, $\sqrt{2M}a(k), \sqrt{2M}b(k)$ are independent standard Gaussian random variables. We apply inequality (\ref{lower}) from Lemma \ref{chi_square} with $c_k = 1$, $k \in \{1,\dots,M\}$, to obtain that, for any $t>0$,
\begin{equation*}
\mathbb{P}\{2M\Vert h\Vert _2^2 \ge \sqrt{8Mt} + 2t + 2M\}\le e^{-t}.
\end{equation*}
\noindent Taking $t= M/2$, we have
\begin{equation}\label{eq_upper_norm_gaussian}
\mathbb{P}\{\Vert h\Vert _2^2 > 4\} = \mathbb{P}\{2M\Vert h\Vert _2^2 > 8M\} \le \mathbb{P}\{2M\Vert h\Vert _2^2 \ge 5M\}\le e^{-M/2}.
\end{equation}
Similarly, by applying inequality (\ref{upper}) from Lemma \ref{chi_square} with $c_k = 1$, we get
\begin{equation*}
\mathbb{P}\left\lbrace\Vert h\Vert _2^2 \le -\sqrt{\frac{2t}{M}} + 1\right\rbrace\le e^{-t},
\end{equation*}
\noindent for every $t>0$. Taking $t= 9M/32$, we obtain
\begin{equation}\label{eq_lower_norm_gaussian}
\mathbb{P}\left\lbrace\Vert h\Vert _2^2 \le -\sqrt{\frac{2t}{M}} + 1\right\rbrace = \mathbb{P}\left\lbrace\Vert h\Vert _2^2 \le \frac{1}{4}\right\rbrace\le e^{-9M/32},
\end{equation}
Summarizing the bounds obtained in \eqref{eq_upper_norm_gaussian} and \eqref{eq_lower_norm_gaussian}, we conclude the desired claim.
\end{proof}

\subsection{Fourier bias}

In additive combinatorics, the notion of Fourier bias is used to measure pseudorandomness of a set. Roughly speaking, it helps to distinguish between sets which are highly uniform and behave like random sets, and those which are highly non-uniform and behave like arithmetic progressions \cite{tao}. 

\begin{definition}
Take $C\subset \mathbb{Z}_M$ and let ${\bf1}_C$ be the characteristic function of $C$. Then the \emph{Fourier bias} of $C$ is given by
\begin{equation*}
\Vert C\Vert _u = \max_{m\in \mathbb{Z}_M\setminus \{0\}}{|(\mathcal{F}_M {\bf1}_C)(m)|}.
\end{equation*}
\end{definition}

The following lemma follows from Chernoff's inequality and can be found in \cite[Lemma~4.16]{tao}. Loosely speaking, it shows that, if $B$ is a random subset of $A\subset \mathbb{Z}_M$, then $\Vert B\Vert _u$ is tightly concentrated around $\frac{|B|}{|A|}\Vert A\Vert _u$. In other words, the Fourier bias of a random subset scales proportionally to its cardinality.

\begin{lemma}\label{lemma_fourier_bias_rand}
Consider an additive subset $A$ of $\mathbb{Z}_M$ with $M>4$, and fix $0<\tau\le 1$. Let $B$ be a random subset of $A$, such that ${\bf 1}_{B} (a)\sim \text{i.i.d.}~B(1,\tau)$, for $a\in A$, that is, events $\{a\in B\}$ are independent and have probability $\tau$. Then, for any $\lambda>0$ and $\sigma^2 = \frac{|A|}{M^2}\tau (1 - \tau)$, we have
\begin{equation*}
\mathbb{P}\left\lbrace \big| \Vert B\Vert _u - \tau \Vert A\Vert _u\big| \ge \lambda \sigma \right\rbrace \le 4M \max \left\lbrace e^{-\frac{\lambda^2}{8}}, ~e^{-\frac{\lambda \sigma}{2\sqrt{2}}} \right\rbrace.
\end{equation*}
\end{lemma}

As an easy consequence of Lemma \ref{lemma_fourier_bias_rand}, we obtain the following result that provides an efficient bound on the absolute value of the sum of randomly sampled roots of unity.

\begin{corollary}\label{cor_sum_rand_roots_of_unity}
Let $B$ be a random subset of $\mathbb{Z}_M$, such that ${\bf 1}_{B} (m)\sim \text{i.i.d.}~B(1,\tau)$, for $m\in \mathbb{Z}_M$ and $0<\tau < 1$. Then, for any constant $C>4\sqrt{2}$, we have
\begin{equation*}
\mathbb{P}\left\lbrace \max_{m\in \mathbb{Z}_M\setminus \{0\}} \left|\sum_{b\in B} e^{2\pi i bm/M}\right| < C\log M \right\rbrace \ge 1 -  \frac{1}{M^{\frac{C}{2\sqrt{2}} - 2}}.
\end{equation*}
\end{corollary}

\begin{proof}
Let us apply Lemma \ref{lemma_fourier_bias_rand} with $A = \mathbb{Z}_M$. Then, since $\Vert \mathbb{Z}_M\Vert _u = 0$ and $\sigma^2 = \frac{|A|}{M^2}\tau (1 - \tau) = \frac{\tau (1 - \tau)}{M}$, for any $\lambda>0$ we obtain
\begin{equation*}
\mathbb{P}\left\lbrace  \Vert B\Vert _u \ge \lambda \sqrt{\frac{\tau (1 - \tau)}{M}} \right\rbrace \le 4M \max \left\lbrace e^{-\frac{\lambda^2}{8}}, ~e^{-\frac{\lambda \sqrt{\tau (1 - \tau)}}{2\sqrt{2M}}} \right\rbrace.
\end{equation*}
Then, by choosing $\lambda = \frac{C}{\sqrt{\tau(1 - \tau)}}\sqrt{M}\log M$ with a constant $C>4\sqrt{2}$, we ensure that
\begin{equation*}
\Scale[0.93]{
\begin{split}
4M \max \left\lbrace e^{-\frac{\lambda^2}{8}}, ~e^{-\frac{\lambda \sqrt{\tau (1 - \tau)}}{2\sqrt{2M}}} \right\rbrace = \max \left\lbrace e^{-\frac{C^2 M\log^2 M}{8\tau(1 - \tau)} + \log (4M)}, ~e^{-\frac{C\log M}{2\sqrt{2}} + \log (4M)}\right\rbrace = \frac{1}{M^{\frac{C}{2\sqrt{2}} - 2}}.
\end{split}}
\end{equation*}
Thus, we obtain that
\begin{equation*}
\mathbb{P}\left\lbrace  \Vert B\Vert _u \ge C\log M \right\rbrace \le \frac{1}{M^{\frac{C}{2\sqrt{2}} - 2}},
\end{equation*}
\noindent and $\Vert B\Vert _u = \max_{m\in \mathbb{Z}_M\setminus \{0\}}{|(\mathcal{F} {\bf1}_B)(m)|} = \max_{m\in \mathbb{Z}_M\setminus \{0\}} \left|\sum_{b\in B} e^{2\pi i bm/M}\right|$, which concludes the proof.
\end{proof}

\end{document}